\documentclass[11pt,a4paper]{article}

\usepackage{amsfonts}
\usepackage{amsthm}
\usepackage{amsmath}
\usepackage{amssymb}
\usepackage{amscd}
\usepackage{mathrsfs}
\usepackage[mathscr]{eucal}
\usepackage{graphicx}
\usepackage{epstopdf} 
\usepackage{pict2e}
\usepackage{epic}
\usepackage{xcolor}
\usepackage{float}
\usepackage{mathtools}
\usepackage{centernot}

\usepackage{cite}
\usepackage{hyperref}
\hypersetup{colorlinks=true, urlcolor= black, linkcolor=black, citecolor=black}

\usepackage[utf8]{inputenc}
\usepackage[T1]{fontenc}
\usepackage{lmodern}
\usepackage{dsfont}
\usepackage{indentfirst}

\numberwithin{equation}{section}

\theoremstyle{plain}
\newtheorem{Thm}{Theorem}[section]
\newtheorem{Lem}[Thm]{Lemma}
\newtheorem{Coro}[Thm]{Corollary}
\newtheorem{Prop}[Thm]{Proposition}

\theoremstyle{definition}
\newtheorem{Def}[Thm]{Definition}

\newtheorem{Rem}[Thm]{Remark}

\usepackage[margin=2.5cm]{geometry}

\newcommand\n{\mathbf{n}}

\newcommand\sn{\partial\mathbf{n}}

\newcommand\fg{\mathfrak{g}}

\newcommand{\connect}{\xleftrightarrow}

\title{Mass scaling of the near-critical Ising model in dimensions $d\geq 4$}
\begin{document}

\author{Romain Panis}
\date{}

\maketitle
\abstract{We study the Ising model on $\mathbb{Z}^d$ with $d\geq 4$ and derive 
near-critical bounds on the truncated two-point function 
$\langle\sigma_0;\sigma_x\rangle_{\beta,h} 
:= \langle\sigma_0\sigma_x\rangle_{\beta,h} 
- \langle\sigma_0\rangle_{\beta,h}\langle\sigma_x\rangle_{\beta,h}$ at parameters $\beta\leq\beta_c$ and $h\geq 0$. As a corollary, we obtain that 
the associated mass (or exponential decay rate) is equal to
\begin{equation*}
	\max\bigl((\beta_c-\beta)^{1/2},h^{1/3}\bigr)^{1+o(1)},
\end{equation*} 
where $o(1)$ tends to $0$ as 
$(\beta,h)$ tends to $(\beta_c,0)$. The proof combines the corresponding result 
at $h=0$, recently established by Duminil-Copin and Panis 
\cite{DuminilPanis2024newLB}, with an interpolation argument inspired by 
Aizenman and Fern\'andez 
\cite{AizenmanFernandezCriticalBehaviorMagnetization1986} and carried out 
via the random current representation of the model.}

\section{Introduction}

The Ising model is one of the most fundamental examples of a discrete statistical mechanics model undergoing a \emph{phase transition}. We refer to \cite{FriedliVelenikIntroStatMech2017,DuminilLecturesOnIsingandPottsModels2019} for mathematical introductions, and to \cite{DuminilreviewIsing} for a review. 

\vspace{4pt}

In this work, we study the Ising model on subgraphs of the hypercubic lattice  $\mathbb Z^d$. The associated edge-set is denoted 
$E(\mathbb{Z}^d):=\{xy:\Vert x-y\Vert_1=1\}$, where $\Vert\cdot\Vert_1$ 
denotes the $\ell^1$ norm on $\mathbb{R}^d$. We write $x\sim y$ if $xy\in E(\mathbb Z^d)$. Let $G=(V,E)$ be a 
finite subgraph of $\mathbb{Z}^d$. For an \emph{interaction} 
$J=(J_e)_{e\in E(\mathbb Z^d)}\in(\mathbb{R}^+)^{E(\mathbb Z^d)}$, an \emph{external magnetic field} 
$h\in\mathbb{R}$, and a \emph{boundary condition} 
$\tau\in\{-1,0,1\}^{\mathbb{Z}^d}$, define a probability measure on 
$\{-1,1\}^V$ by setting
\begin{equation}
    \langle F(\sigma)\rangle_{G,J,h}^\tau
    :=\frac{1}{Z^\tau_{G,J,h}}\sum_{\sigma\in\{-1,1\}^V}F(\sigma)
    \exp\Big(\sum_{xy\in E}J_{xy}\sigma_x\sigma_y
    +h\sum_{x\in V}\sigma_x
    +\sum_{\substack{x\in V\\y\notin V}}
    \mathds{1}_{x\sim y}J_{xy}\sigma_x\tau_y\Big),
\end{equation}
where $F:\{-1,1\}^V\rightarrow\mathbb{R}$ and $Z^{\tau}_{G,J,h}$ is the 
\emph{partition function}, ensuring that $\langle\cdot\rangle_{G,J,h}^\tau$ 
is a probability measure. The choice $\tau\equiv 0$ (resp. $\tau\equiv 1$) corresponds to the 
\emph{free} (resp. \emph{plus}) measure, which we denote by 
$\langle\cdot\rangle_{G,J,h}$ (resp. $\langle\cdot\rangle_{G,J,h}^+$). 
It is classical 
(see~\cite{GriffithsCorrelationsIsing1-1967,GriffithsCorrelationsIsing2-1967}) 
that both families of measures admit an infinite-volume weak limit as 
$G\nearrow\mathbb{Z}^d$, denoted by $\langle\cdot\rangle_{J,h}$ and 
$\langle\cdot\rangle_{J,h}^+$ respectively. When $h=0$, we omit it from 
the notation.

We focus on the homogeneous case $J\equiv \beta$, where $\beta\geq 0$ is called the \emph{inverse temperature}. When $h=0$, the Ising model undergoes a phase transition as $\beta$ varies. In dimensions $d\geq 2$, there exists $\beta_c\in (0,\infty)$ \cite{PeierlsIsing1936} such that the following holds: letting $m^*(\beta):=\langle \sigma_0\rangle_{\beta}^+$ denote the \emph{spontaneous magnetisation}, one has
\begin{equation}\label{eq: mag}
	m^*(\beta)\begin{cases}\displaystyle=0 &\text{ if }\beta<\beta_c,\\
    \displaystyle>0&\text{ if }\beta>\beta_c.\end{cases}
\end{equation}
This phase transition is \emph{continuous} (or \emph{second-order}) in the sense that $m^*(\beta_c)=0$, see \cite{Yang1952spontaneous,AizenmanFernandezCriticalBehaviorMagnetization1986,WernerPercolationEtModeledIsing2009,AizenmanDuminilSidoraviciusContinuityIsing2015}. The change of behaviour in \eqref{eq: mag} is related to the possibility of \emph{breaking the flip symmetry} of the system, i.e.~to the existence of a translation-invariant infinite-volume measure distinct from $\langle \cdot \rangle_{\beta}$. For $\beta\leq \beta_c$, the free measure is the only such measure; while for $\beta>\beta_c$, one has $\langle \cdot\rangle_{\beta}^+\neq \langle \cdot\rangle_{\beta}$ (in fact much more is known \cite{BodineauTranslationGibbsIsing2006,RaoufiGibbsMeasures}). When $h>0$, the symmetry is always broken: $\langle \sigma_0\rangle_{\beta,h}>0$. Moreover, there exists exactly one infinite volume measure (see \cite{FriedliVelenikIntroStatMech2017} and references therein) so that $\langle \cdot\rangle_{\beta,h}=\langle \cdot\rangle_{\beta,h}^+$. 

\vspace{4pt}

We are interested in the behaviour of the model for $(\beta,h)$ close to 
the critical point $(\beta_c,0)$. The picture changes drastically, both quantitatively and qualitatively, depending on the dimension.

The \emph{upper-critical dimension} of the model is expected to be equal to four (see e.g.~\cite{WilsonFisher1972dcexpansion}). This means that the model should exhibit \emph{mean-field behaviour} in dimensions $d\geq 4$ but not in dimensions $2\leq d<4$.
The \emph{mean-field regime} $d\geq 4$ is characterised by a \emph{Gaussian} 
behaviour, reflected in \emph{trivial} (near-)critical scaling limits 
and \emph{critical exponents} matching those of the model defined on a 
tree or the complete graph. Several methods have been developed to rigorously justify these 
predictions. The most notable ones are the \emph{lace expansion}, 
introduced by Brydges and Spencer~\cite{BrydgesSpencerSAW} and later 
applied to the Ising model in dimensions $d\gg 4$ 
\cite{Sakai2007LaceExpIsing,Sakai2022correctboundsIsing}, and the 
differential inequality approach pioneered by Aizenman in the 
1980s~\cite{AizenmanGeometricAnalysis1982} which exploits a special 
property of the model called \emph{reflection positivity} \cite{FrohlichSimonSpencerIRBounds1976,FrohlichIsraelLiebSimon1978}.  Thanks to these developments, the (near-)critical phase of the Ising model in dimensions $d\geq 4$ is now relatively well-understood, especially in the zero external magnetic field case. Main successes include triviality results \cite{AizenmanGeometricAnalysis1982,FrohlichTriviality1982,AizenmanDuminilTriviality2021,PanisTriviality2023} and critical exponent computations \cite{AizenmanFernandezCriticalBehaviorMagnetization1986,AizenmanBarskyFernandezSharpnessIsing1987,AizenmanGeometricAnalysis1982,FrohlichTriviality1982,DuminilPanis2024newLB,vEGPS,AizenmanGrahamRenormalizedCouplingSusceptibility4d1983,Sakai2007LaceExpIsing}.
Additional approaches to the study of the mean-field regime have been developed in related models, 
see~\cite{DumPan24WSAW,DumPan24Perco,duminil2026random,hutchcroft2025criticalI,hutchcroft2025criticalIII}. Moreover, the rigorous \emph{renormalization group} method has been developed in the close context of the weakly coupled $\varphi^4$ model, see e.g.~\cite{gawedzki1985massless,BBS19,SladeTombergRGWeakPhi4in2016,BauerschmidtBrydgesSlade2014Phi4fourdim,lohmann2020critical,park2025renormalisation,bauerschmidt2017finite}.

It is known that the Ising model is not mean-field when $d=2$, 
see e.g.,~\cite{AizenmanGeometricAnalysis1982,camia2016planar}. The
two-dimensional model is now very well understood, owing to special properties 
including \emph{exact integrability} 
\cite{Onsager1944,Yang1952spontaneous,McCoyWu1973two} and 
\emph{conformal invariance} 
\cite{SmirnovConformalInvarianceIsing,ChelkakHonglerIz2015conformal,ChelkakDCHonglerKempSmir2014CV}. 
The three-dimensional Ising model, by contrast, lacks appropriate tools for its analysis. Nevertheless, the \emph{conformal bootstrap} approach yields 
several precise predictions on its critical exponents 
\cite{RychkovNonGaussianity2017,ElShowkRychkov2012ConfBootstrapI,ElShowkRychkov2014ConfBootstrapII}.

\vspace{4pt}


\vspace{4pt}

A fundamental quantity in the analysis of the Ising model is the so-called \emph{truncated two-point function} defined, for $\beta,h\geq 0$ and $x,y\in \mathbb Z^d$ by
\begin{equation}
	\langle \sigma_x;\sigma_y\rangle_{\beta,h}^+:=\langle \sigma_x\sigma_y\rangle_{\beta,h}^+-\langle \sigma_x\rangle_{\beta,h}^+\langle \sigma_y\rangle_{\beta,h}^+.
\end{equation}
As proved by Graham~\cite{graham1982correlation}, the truncated two-point 
function satisfies the following inequality: for every $x,y,z\in\mathbb{Z}^d$,
\begin{equation}
    \langle\sigma_x;\sigma_y\rangle_{\beta,h}^+
    \geq\langle\sigma_x;\sigma_z\rangle_{\beta,h}^+
    \langle\sigma_z;\sigma_y\rangle_{\beta,h}^+.
\end{equation}
By a classical application of Fekete's lemma, one obtains the existence 
of the \emph{correlation length} $\xi(\beta,h)$, defined by
\begin{equation}\label{eq:def correlation length intro}
    \xi(\beta,h)
    :=\Big(
       -\lim_{n\to\infty}\frac{1}{n}
       \log\langle\sigma_0;\sigma_{n\mathbf{e}_1}\rangle_{\beta,h}^+
    \Big)^{-1}\in [0,\infty],
\end{equation}
where $\mathbf{e}_1=(1,0,\ldots,0)\in \mathbb Z^d$, but also that
\begin{equation}\label{eq:bound truncated fekete}
	\langle \sigma_0;\sigma_{n\mathbf{e}_1}\rangle_{\beta,h}^+\leq \exp\left(-\frac{n}{\xi(\beta,h)}\right).
\end{equation}
The \emph{mass} of the model is given by $\xi(\beta,h)^{-1}$.
It is classical that $\xi(\beta,h)\in(0,\infty)$ when $\beta,h>0$, see e.g.~\cite[Theorem~1.2]{ott2020sharp}. Moreover, one has that $\xi(\beta,0)\in(0,\infty)$ 
when $\beta\in (0,\infty)\setminus\{\beta_c\}$~\cite{AizenmanBarskyFernandezSharpnessIsing1987,DuminilTassionNewProofSharpness2016,duminil2020exponential}. Finally, it holds that $\xi(\beta_c,0)=\infty$ by the combination of \cite{SimonInequalityIsing1980} and \cite{AizenmanDuminilSidoraviciusContinuityIsing2015}, meaning that correlations persist over all length scales. More precisely, it is predicted \cite{Widom1965equation} that there exists a critical exponent $\eta$, called the \emph{anomalous dimension}, such that
\begin{equation}\label{eq:def eta}
	\langle \sigma_0\sigma_x\rangle_{\beta_c,0}^+=\langle \sigma_0\sigma_x\rangle_{\beta_c}=\frac{1}{|x|^{d-2+\eta+o(1)}},
\end{equation}
where $|\cdot|$ denotes the infinite norm, and where $o(1)$ tends to $0$ as $|x|$ tends to infinity. 

For a continuous phase transition, one expects $\xi(\beta,h)$ to 
diverge as $(\beta,h)$ approaches $(\beta_c,0)$. Here and below, we restrict the discussion to the regime $\beta\leq \beta_c$ and $h\geq 0$. In particular, by the uniqueness results recalled above, one has that $\langle \cdot\rangle^+_{\beta,h}=\langle \cdot\rangle_{\beta,h}$. It is predicted that there exist critical exponents $\nu,\nu'>0$ such that
\begin{equation}\label{eq: postulate of divergence for xi}
	\xi(\beta,h)=\min\bigl( (\beta_c-\beta)^{-\nu +o(1)},h^{-\nu'+o(1)}\bigr),
\end{equation}
where $o(1)$ goes to $0$ as $(\beta,h)$ tends to $(\beta_c,0)$.
The $\min$ in \eqref{eq: postulate of divergence for xi} reflects 
the fact that, away from $(\beta_c,0)$, the dominant ``off-critical'' effect comes 
either from the thermal perturbation ($\beta<\beta_c$) or from the 
magnetic perturbation ($h>0$). This gives rise, respectively, to what we call the \emph{thermal} and 
\emph{magnetic} off-critical regimes. 

On the physics side, the values of $\nu$ and $\nu'$ can be computed using \emph{scaling theory}. Indeed, \eqref{eq:def eta} and \eqref{eq: postulate of divergence for xi} are illustrations of the broader prediction of Widom \cite{Widom1965equation} that, for second-order phase transitions and near or at the critical point, thermodynamic quantities should adopt an algebraic scaling governed by critical exponents. As observed in the 1960s \cite{Fisher1964correlation,EssamFisher1963pade}, these exponents are constrained by \emph{scaling relations}, which allow one to deduce some from the knowledge of others. In particular, the exponents $\nu$ and $\nu'$ are related to other known critical exponents, defined below in the context of the Ising model. For $\beta\leq \beta_c$ and $h\geq 0$, let
\begin{equation}
	\chi(\beta,h):=\sum_{x\in \mathbb Z^d}\langle \sigma_0;\sigma_x\rangle_{\beta,h}, \qquad M(\beta,h):=\langle \sigma_0\rangle_{\beta,h},
\end{equation}
denote the \emph{truncated susceptibility} and the \emph{magnetisation}, respectively. The quantity $\chi(\beta,h)$ (resp. $M(\beta_c,h)$) is expected to diverge (resp. vanish) as $(\beta,h)\to (\beta_c,0)$ (resp. as $h\searrow 0$), leading to the introduction of critical exponents $\gamma,\gamma',\delta>0$ defined by
\begin{equation}\label{eq: def gamma delta}
	\chi(\beta,h)=\min\bigl((\beta_c-\beta)^{-\gamma+o(1)},h^{-\gamma'+o(1)}\bigr), \qquad M(\beta_c,h)=h^{1/\delta+o(1)},
\end{equation}
where, as before, $o(1)$ tends to $0$ as $(\beta,h)$ approaches $(\beta_c,0)$. The exponents introduced in \eqref{eq:def eta}, \eqref{eq: postulate of divergence for xi}, and \eqref{eq: def gamma delta} are predicted to satisfy the following relations (see \cite{Fisher1998renormalization,DuminilreviewIsing} and references therein):
\begin{equation}\label{eq:scaling relation}
	2-\eta=\frac{\gamma}{\nu}=\frac{\gamma'}{\nu'}=\min(d,d_c)\frac{\delta-1}{\delta+1},
\end{equation}
where $d_c$ is the upper-critical dimension of the model. For the Ising model, it is conjectured that $d_c=4$. Moreover, in dimensions $d\geq 4$, it is known that $\gamma=1$ when $h=0$ (in the sense that $\chi(\beta,0)=(\beta_c-\beta)^{-1+o(1)}$), see \cite{AizenmanGeometricAnalysis1982,AizenmanGrahamRenormalizedCouplingSusceptibility4d1983}; but also that $\delta=3$, and $\gamma'=2/3$ when $\beta=\beta_c$ (in the sense that $\chi(\beta_c,h)=h^{-2/3+o(1)}$), see \cite{AizenmanFernandezCriticalBehaviorMagnetization1986}. Substituting these values in the above scaling relation leads to the prediction $(\nu,\nu')=(1/2,1/3)$. Additionally, we note that \cite{DuminilPanis2024newLB} proves that $\eta=0$ and $\nu=1/2$ in the regime where $h=0$ (in the sense that $\xi(\beta,0)=(\beta_c-\beta)^{-1/2+o(1)}$).
\vspace{4pt}

The main purpose of the present article is to estimate $\xi(\beta,h)$ in dimensions $d\geq 4$ and when $\beta\leq\beta_c$ and $h\geq 0$. We prove that \eqref{eq: postulate of divergence for xi} holds with $\nu=1/2$ and $\nu'=1/3$, thereby confirming the prediction from scaling theory. Our results also give that the scaling of $\chi(\beta,h)$ in \eqref{eq: def gamma delta} holds with $\gamma=1$ and $\gamma'=2/3$, in the same regime of parameters. Assuming that $d_c=4$, this gives a proof of \eqref{eq:scaling relation} in dimensions $d\geq 4$. Before stating the main results, we briefly review the \emph{near-critical regime} of the model.

\vspace{4pt}

\noindent\textbf{Notations.}
Let $\mathbb N:=\{0,1,2,\ldots\}$. If $a,b\in \mathbb R$, we let $a\wedge b:=\min(a,b)$. For $x=(x_1,\ldots,x_d)\in\mathbb R^d$, let $|x|:=\max_{1\leq i\leq d}|x_i|$ denote
the infinite norm. If $n\in \mathbb N$, let $\Lambda_n:=\{x\in \mathbb Z^d: |x|\leq n\}$. For $S\subset\mathbb Z^d$, let
$E(S):=\{xy\in E(\mathbb Z^d):x,y\in S\}$, $\partial S:=\{x\in S: \exists y\notin S, \: x\sim y\}$, and $\textup{diam}(S):=\max\{|x-y|: x,y\in S\}$. In what follows, we often identify the set $S$ with the subgraph of $\mathbb Z^d$ $(S,E(S))$. We denote by $\mathbf{e}_i$ the element of the canonical basis of $i$-th coordinate equal to $1$. If $f,g>0$, we write $f\lesssim g$ (or $g\gtrsim f$) if there exists $C=C(d)>0$ such that $f\leq Cg$. If $f\lesssim g$ and $g\lesssim f$, we write $f\asymp g$.

\subsection{The near-critical regime}

We will restrict to the quarter-plane of parameters $0\leq \beta\leq \beta_c$ and $h\geq 0$. This motivates the following definitions: 
\begin{equation}
    \mathcal Q:=\{(\beta,h):0\leq\beta\leq\beta_c,\ h\geq0\}, \qquad \mathcal Q^*:=\mathcal Q\setminus \{(\beta_c,0)\}.
\end{equation}
Let $(\beta,h)\in \mathcal Q^*$.
By the uniqueness results recalled above, one has $\langle\cdot\rangle_{\beta,h}^+ 
= \langle\cdot\rangle_{\beta,h}$. This means that we can replace the plus measure by the free one in \eqref{eq:def correlation length intro}. Understanding the behaviour of 
$\xi(\beta,h)$ as $(\beta,h)\to(\beta_c,0)$ requires estimates on the 
truncated two-point function that are uniform in both $(\beta,h)$ and the 
spatial scale $|x|$. It is predicted that $\langle\sigma_0;\sigma_x\rangle_{\beta,h}$ 
exhibits three spatial regimes, determined by how $|x|$ compares to 
$\xi(\beta,h)$.
Below, we omit the precise meaning of $\ll,\gg,\approx$ for simplicity.
\begin{itemize}
    \item If $|x|\ll\xi(\beta,h)$, the truncated two-point function 
    is of order $\langle \sigma_0\sigma_x\rangle_{\beta_c}$, 
    reflecting the fact that the model still ``feels critical''.

    \item If $|x|\gg\xi(\beta,h)$, the truncated two-point function is 
    \emph{massive}: its spatial 
    decay is described by the \emph{Ornstein--Zernike theory}~\cite{ornstein1914accidental,zernike1916clustering}, 
    characterised by exponential decay at a direction-dependent rate 
    (equal to $\xi(\beta,h)^{-1}$ in the $\mathbf{e}_1$ direction) and 
    a prefactor of order $|x|^{-(d-1)/2}$. 

    \item If $|x|\approx\xi(\beta,h)$, the picture is more complex: 
    neither of the two preceding descriptions is accurate. This is the 
    \emph{near-critical} regime, where one observes a \emph{crossover} between the above two distinct types of decay.
\end{itemize} A customary scaling
ansatz (see e.g.~\cite{Fisher1998renormalization}) unifies these three regimes in the form 
\begin{equation}\label{eq:expected behaviour in the near-critical regime}
    \langle\sigma_0;\sigma_x\rangle_{\beta,h}
    \approx
    \frac{1}{|x|^{d-2+\eta}}\,
    \mathcal F\left(\frac{x}{\xi(\beta,h)}\right),
\end{equation} 
where $\mathcal F$ remains of constant order near the origin and decays
exponentially at infinity, and $\eta$ is the same exponent as in \eqref{eq:def eta}.
%

The regime where $|x|$ tends to infinity at a fixed $(\beta,h)$ is well-understood. There, exact asymptotic estimates on the truncated two-point function were obtained using the rigorous Ornstein--Zernike (OZ) theory initiated in \cite{campanino2002ornstein}: the case $\beta<\beta_c$ and $h=0$ was studied in \cite{campanino2003ornstein}, while the setting $h>0$ and $\beta\geq 0$ was treated in \cite{ott2020sharp}. However, these works are not quantitative in $(\beta,h)$ and do not identify $\xi(\beta,h)$ as the scale beyond which the OZ estimate starts being true.

Other known results focus on the critical and near-critical regimes. They are restricted to the two-dimensional setting and the 
mean-field regime $d\geq 4$. 

For the two-dimensional Ising model, it is known that $\eta=1/4$ \cite{McCoyWu1973two}, and that $\xi(\beta,0)=(\beta_c-\beta)^{-1+o(1)}$ \cite{duminil2014near,McCoyWu1973two} where $o(1)$ tends to $0$ as $\beta\nearrow \beta_c$. Near-critical bounds on $\langle \sigma_0;\sigma_x\rangle_{\beta_c,h}$ were derived in \cite{CamiaJiangNewmanNearCritical2020,camia2020fk} (see also \cite{KlausenRaoufi2022,jiang2025simple} for shorter proofs of the upper bound). These bounds imply that $\xi(\beta_c,h)\asymp h^{-8/15}$ as $h\searrow 0$. Let us mention that these results also yield interesting information regarding the near-critical scaling limit of the model constructed in \cite{camia2016planar}. Finally, in a recent work \cite{dalimonte2025near}, a version of \eqref{eq:expected behaviour in the near-critical regime} in the case $h=0$ (and holding more generally for the \emph{random cluster model} with cluster weight $q\in [1,4)$) was derived, identifying up to a multiplicative constant the function $\mathcal F$.

 In dimensions $d\geq 4$, \cite{DuminilPanis2024newLB} proves that $\eta=0$ and derives near-critical bounds on the two-point function at $(\beta,0)$ with $\beta\leq \beta_c$. In particular, it is known that $\xi(\beta,0)=(\beta_c-\beta)^{-1/2+o(1)}$ where $o(1)$ tends to $0$ as $\beta\nearrow \beta_c$. Results in the magnetised setting are more ``averaged''. As mentioned above, \cite{AizenmanFernandezCriticalBehaviorMagnetization1986} obtains that
 \begin{equation}
 	\chi(\beta_c,h)=\sum_{x\in \mathbb Z^d}\langle \sigma_0;\sigma_x\rangle_{\beta_c,h}\asymp h^{-2/3+o(1)},
 \end{equation}
 with $o(1)$ tending to $0$ as $h$ tends to $0$.
 If \eqref{eq:expected behaviour in the near-critical regime} holds true, one obtains that $\chi(\beta_c,h)\asymp \xi(\beta_c,h)^{2+o(1)}$, recovering the prediction from scaling theory: 
 \begin{equation}
 \xi(\beta_c,h)\asymp h^{-1/3+o(1)}.
 \end{equation}
  In this paper, we derive near-critical bounds on $\langle\sigma_0;\sigma_x\rangle_{\beta_c,h}$ that are sufficient to compute the $1/3$ exponent. We do not, however, prove any version of \eqref{eq:expected behaviour in the near-critical regime}. Our techniques allow us to go beyond the case $\beta=\beta_c$ and provide the scaling of $\xi(\beta,h)$ in the whole quarter plane $\mathcal Q^*$.
 
Let us mention that the above questions have been studied in other contexts of statistical mechanics. Near-critical bounds on the two-point function have been derived for (spread-out) Bernoulli percolation in dimensions $d>6$ \cite{HutchcroftMichtaSladePercolationTorusPlateau2023,DumPan24Perco,duminil2026random,panis2026reversing}, the weakly self-avoiding walk model \cite{Liu2023Plateau,Slade2023Plateau,DumPan24WSAW} in dimensions $d>4$, and spread-out lattice trees \cite{liu2025near} in dimensions $d>8$. In a recent paper, Liu and Slade \cite{liu2026crossover} used the lace expansion to prove \eqref{eq:expected behaviour in the near-critical regime} (in the strongest sense) for the self-avoiding walk $(d>4)$ and Bernoulli percolation ($d\geq 15$). We expect that their methods can be combined with the works of Sakai \cite{Sakai2007LaceExpIsing,Sakai2022correctboundsIsing} to obtain the corresponding result for the Ising model in dimensions $d\gg 4$.

\subsection{Statement of the results}

We now describe the main results of this paper. Recall that $\mathcal Q=\{(\beta,h):0\leq\beta\leq\beta_c,\ h\geq0\}$ and $\mathcal Q^*=\mathcal Q\setminus \{(\beta_c,0)\}$, and define, for $\delta\in (0,\beta_c)$,
\begin{equation}
    \mathcal Q_\delta
    :=\{(\beta,h):\beta_c-\delta\leq\beta\leq\beta_c,\
    0\leq h\leq\delta\}\setminus \{(\beta_c,0)\}.
\end{equation}
For $(\beta,h)\in \mathcal Q^*$, we let
\begin{equation}
	\ell(\beta,h):=(\beta_c-\beta)^{-1/2}\wedge h^{-1/3}.
\end{equation}
The thermal (resp. magnetic) off-critical regime corresponds to $\ell(\beta,h)=(\beta_c-\beta)^{-1/2}$ (resp. $\ell(\beta,h)=h^{-1/3}$).

Our first result holds in dimensions $d>4$. We obtain a uniform near-critical
upper bound on $\langle \sigma_0;\sigma_x\rangle_{\beta,h}$, together with a matching
order lower bound below the scale $\ell(\beta,h)$.

\begin{Thm}\label{thm:main}
Let $d>4$.  There exist $c,C,\delta>0$ such that, for every
$(\beta,h)\in\mathcal Q_\delta$ and every $x\in\mathbb Z^d$,
\begin{equation}\label{eq:main upper intro}
    \langle\sigma_0;\sigma_x\rangle_{\beta,h}
    \leq \frac{C}{1\vee|x|^{d-2}}
    \exp\left(-c\frac{|x|}{\ell(\beta,h)}\right).
\end{equation}
Moreover, if $|x|\leq c\,\ell(\beta,h)$, then
\begin{equation}\label{eq:main lower intro}
    \langle\sigma_0;\sigma_x\rangle_{\beta,h}
    \geq \frac{c}{1\vee|x|^{d-2}}.
\end{equation}
\end{Thm}

At the upper-critical dimension $d=4$, we obtain the following result.

\begin{Thm}\label{thm:main bis}
Let $d=4$ and $\varepsilon>0$.  There exist $c,C,\delta>0$ such that, for every
$(\beta,h)\in\mathcal Q_\delta$ and every $x\in\mathbb Z^4$,
\begin{equation}\label{eq:nc up bound thm bis}
    \langle\sigma_0;\sigma_x\rangle_{\beta,h}
    \leq \frac{C}{1\vee|x|^2}
    \exp\left(
        -c\frac{|x|}
        {\ell(\beta,h)(\log\ell(\beta,h))^{2+\varepsilon}}
    \right).
\end{equation}
Moreover, if
$|x|\leq c\,\ell(\beta,h)(\log\ell(\beta,h))^{-2}$, then
\begin{equation}\label{eq:main lower four intro}
    \langle\sigma_0;\sigma_x\rangle_{\beta,h}
    \geq \frac{c}{1\vee|x|^2\log|x|},
\end{equation}
with the convention that $|x|^2\log|x|=0$ when $x=0$.
\end{Thm}

As mentioned above, one has $\eta=0$ in dimensions $d\geq 4$ (where $\eta$ is defined in \eqref{eq:def eta}). In fact, \cite{DuminilPanis2024newLB} proves that $\langle \sigma_0\sigma_x\rangle_{\beta_c}\asymp |x|^{2-d}$ when $d>4$, and $|x|^{-2}\gtrsim \langle \sigma_0\sigma_x\rangle_{\beta_c}\gtrsim (1\vee |x|^2\log |x|)^{-1}$ when $d=4$. In particular, the polynomial terms in \eqref{eq:main upper intro}--\eqref{eq:main lower four intro} are of order the critical two-point function $\langle \sigma_0\sigma_x\rangle_{\beta_c}$. As a result, in dimensions $d>4$, the truncated two-point function $\langle \sigma_0;\sigma_x\rangle_{\beta,h}$ looks critical below $\ell(\beta,h)$, and is off-critical---with an exponential rate of order $\ell(\beta,h)^{-1}$---above $\ell(\beta,h)$. A similar picture holds in $d=4$ up to logarithmic corrections.
\begin{Rem}\hspace{-1pt}(1) It is natural (see e.g.~\cite{BauerschmidtBrydgesSlade2014Phi4fourdim,bauerschmidt2017finite}) to see logarithmic corrections appearing in the \emph{marginal} case $d=4$. We do not, however, expect the logarithmic terms in Theorem
\ref{thm:main bis} to be of the correct order. 

(2) The need for a parameter $\varepsilon>0$ is technical, as can be seen from \eqref{eq: where eps becomes crucial}.

(3)  In the thermal off-critical regime
$\ell(\beta,h)=(\beta_c-\beta)^{-1/2}$, the exponent $2+\varepsilon$ on the logarithm in
\eqref{eq:nc up bound thm bis} can be replaced with $1$.  Indeed, the inequality
    $\langle\sigma_0;\sigma_x\rangle_{\beta,h}
    \leq\langle\sigma_0\sigma_x\rangle_{\beta}$
(see Proposition \ref{prop: trunc is bounded by h=0}) reduces the upper bound in that regime to the zero-field result of
\cite{DuminilPanis2024newLB} stated in Theorem
\ref{thm:bounds 2pt d=4}.	
\end{Rem}

The above results can be used to derive estimates on the correlation length $\xi(\beta,h)$ defined in \eqref{eq:def correlation length intro} (and therefore on the mass $\xi(\beta,h)^{-1}$). Again, additional logarithmic corrections emerge when $d=4$.

\begin{Coro}\label{coro:correlation length intro}
Let $d\geq 4$ and $\varepsilon>0$. There exist $c,C,\delta>0$ such that, for every $(\beta,h)\in \mathcal Q_\delta$,
\begin{equation}\label{eq:correlation length bounds intro}
    c\,\ell(\beta,h)(\log\ell(\beta,h))^{-1}
    \leq \xi(\beta,h)\leq C\ell(\beta,h), \qquad \textup{if }d>4,
\end{equation}
and 
\begin{equation}\label{eq:correlation length bounds intro bis}
    c\,\ell(\beta,h)(\log\ell(\beta,h))^{-3}
    \leq \xi(\beta,h)\leq C\ell(\beta,h)(\log \ell(\beta,h))^{2+\varepsilon}, \qquad \textup{if }d=4.
\end{equation}
\end{Coro}
In fact, we expect $\xi(\beta,h)\asymp \ell(\beta,h)$ in dimensions $d>4$. In dimension $d=4$, we conjecture\footnote{Using the results of \cite{bauerschmidt2017finite} 
and assuming that the $\varphi^4$ and Ising models lie in the 
same universality class leads to the prediction $\tilde{\nu}=1/6$.} the existence of $\tilde \nu,\tilde \nu'>0$ such that 
\begin{equation}
	\xi(\beta,h)=(\beta_c-\beta)^{-1/2}|\log (\beta_c-\beta)|^{\tilde \nu+o(1)}\wedge h^{-1/3}|\log h|^{\tilde \nu'+o(1)},
	\end{equation}
	where $o(1)$ tends to $0$ as $(\beta,h)\to (\beta_c,0)$. 
	
	The proof of Corollary \ref{coro:correlation length intro} is a direct consequence of Theorems \ref{thm:main}--\ref{thm:main bis} and \eqref{eq:bound truncated fekete}.
\begin{proof}[Proof of Corollary~\textup{\ref{coro:correlation length intro}}] We begin with the case $d>4$. Let $c,C,\delta>0$ be given by Theorem \ref{thm:main}. On the one hand, combining \eqref{eq:main upper intro} with the definition of $\xi(\beta,h)$ gives that, for every $(\beta,h)\in \mathcal Q_\delta$, one has 
\begin{equation}
\xi(\beta,h)\leq c^{-1}\ell(\beta,h).
\end{equation}
On the other hand, let us assume by contradiction that $\xi(\beta,h)< c_1\ell(\beta,h)(\log \ell(\beta,h))^{-1}$ for some $c_1\in (0,c)$ to be chosen small enough below. Combining \eqref{eq:bound truncated fekete} and \eqref{eq:main lower intro} gives
\begin{equation}
	\frac{c}{(c\ell(\beta,h))^{d-2}}\leq \langle \sigma_0;\sigma_{c\ell(\beta,h)\mathbf{e}_1}\rangle_{\beta,h}\leq \exp\left(-\frac{c\ell(\beta,h)}{\xi(\beta,h)}\right)\leq \frac{1}{\ell(\beta,h)^{c/c_1}},
\end{equation}
where the last inequality follows from our assumption. Clearly, the above cannot hold if $c_1$ is small enough, uniformly in $\ell(\beta,h)\geq 2$. 

We now assume that $d=4$. Let $c',C',\delta'>0$ be given by Theorem \ref{thm:main bis}. Arguing exactly as above, one obtains that, for every $(\beta,h)\in \mathcal Q_{\delta'}$
\begin{equation}
	\xi(\beta,h)\leq (c')^{-1}\ell(\beta,h)(\log \ell(\beta,h))^{2+\varepsilon}.
\end{equation}
Now, assume that $\xi(\beta,h)<c_2 \ell(\beta,h)(\log \ell(\beta,h))^{-3}$ for some $c_2\in (0,c')$ to be chosen small enough below. Equations
 \eqref{eq:bound truncated fekete} and \eqref{eq:main lower four intro} imply that
\begin{multline}
	\frac{c'(\log \ell(\beta,h))^3}{(c'\ell(\beta,h))^2}\leq \langle \sigma_0;\sigma_{c'\ell(\beta,h)(\log \ell(\beta,h))^{-2}\mathbf{e}_1}\rangle_{\beta,h}\\\leq \exp\left(-c'\frac{\ell(\beta,h)(\log \ell(\beta,h))^{-2}}{\xi(\beta,h)}\right)\leq \frac{1}{\ell(\beta,h)^{c'/c_2}},
\end{multline}
where the last inequality follows from our assumption. Again, the above cannot hold if $c_2$ is small enough, uniformly in $\ell(\beta,h)\geq 2$. This concludes the proof.
\end{proof}

Finally, Theorems \ref{thm:main} and \ref{thm:main bis} allow to analyse $\chi(\beta,h)$. Let us mention that the upper bounds were already known as explained in Theorem \ref{thm: susceptibility} below.
\begin{Coro} Let $d\geq 4$. There exist $c,C,\delta>0$ such that the following holds. For every $(\beta,h)\in \mathcal Q_\delta$,
\begin{equation}
c\,\ell(\beta,h)\leq \chi(\beta,h)\leq C\ell(\beta,h)^2,\qquad \textup{if }d>4,
\end{equation}
and
\begin{equation}
c\,\ell(\beta,h)^2(\log \ell(\beta,h))^{-1}\leq \chi(\beta,h)\leq C\ell(\beta,h)^2\log \ell(\beta,h),\qquad \textup{if }d=4.
\end{equation}
\end{Coro}
\begin{proof} The result follows immediately after plugging the bounds of Theorems \ref{thm:main} and \ref{thm:main bis} in the definition of $\chi(\beta,h)$.
\end{proof}

\subsection{Strategy of proof}

We conclude the introduction with a brief description of the main ideas 
of the paper. For simplicity, we restrict to dimensions $d>4$. We first focus on the upper bound, which is the hardest part and the main 
contribution of the paper. The lower bound is derived via a 
differentiation argument described below.

\vspace{4pt}

One natural approach to derive \eqref{eq:main upper intro} is to adapt the argument of \cite{DuminilPanis2024newLB}, which covers the case $h=0$. Let us briefly describe it. The strategy of \cite{DuminilPanis2024newLB} relies on a fundamental result: the \emph{Simon--Lieb inequality} \cite{SimonInequalityIsing1980,LiebImprovementSimonInequality}.

\begin{Lem}[Simon--Lieb inequality]\label{lem:SL} Fix $S\subset \Lambda\subset \mathbb Z^d$ such that $0\in S$. Let $J\in (\mathbb R^+)^{E(\mathbb Z^d)}$. For every $x\in \Lambda$,
\begin{equation}
	\langle \sigma_0\sigma_x\rangle_{\Lambda,J}\leq  \langle \sigma_0\sigma_x\rangle_{S,J}+\sum_{\substack{u\in S\\ v\in \Lambda\setminus S\\u\sim v}}\langle \sigma_0\sigma_u\rangle_{S,J}\,J_{uv} \,\langle \sigma_v\sigma_x\rangle_{\Lambda,J},
\end{equation}
where we recall that we have identified $S$ (resp. $\Lambda$) with the graph $(S,E(S))$ (resp. $(\Lambda,E(\Lambda))$), and where we used the convention that $\langle \sigma_0\sigma_x\rangle_{S,J}=0$ if $x\notin S$.
\end{Lem}
A natural quantity emerges from the above statement: if $S\subset \mathbb Z^d$ contains $0$, define
\begin{equation}\label{eq: def phi}
	\varphi_J(S):= \sum_{\substack{u\in S\\ v\notin S\\u\sim v}}\langle \sigma_0\sigma_u\rangle_{S,J}\,J_{uv}.
\end{equation}
Let $\beta<\beta_c$. An easy corollary of (iterations of) Lemma \ref{lem:SL} is that, if there exists $S$ finite and containing $0$ such that $\varphi_\beta(S)< \tfrac{1}{2}$, then the following holds: letting $n:=\textup{diam}(S)$ and $k:=\lfloor \tfrac{|x|}{(n+1)}\rfloor$, one has $c,C>0$ (independent of $x$) such that 
\begin{equation}
	\langle \sigma_0\sigma_x\rangle_{\beta}\leq C2^{-k}.
\end{equation}
Such a mechanism is central in the alternative proof of \emph{sharpness} of the phase transition of Duminil-Copin and Tassion \cite{DuminilTassionNewProofSharpness2016}. From the above, we obtain that the correlation length at $\beta$ is \emph{at most} of order $n$. In \cite{DuminilPanis2024newLB}, it is shown that one can pick $n\lesssim (\beta_c-\beta)^{-1/2}$ above. This is done by proving a new correlation inequality, see \cite[Theorem~1.2]{DuminilPanis2024newLB}. The polynomial prefactor of the correct order can be obtained with a mild adaptation of this argument. In the setting $h>0$, it is natural to ask whether a version of Lemma~\ref{lem:SL} holds with truncated two-point functions everywhere. Such a result is not known (though a partial statement 
in this direction was obtained in~\cite{graham1982correlation}). In fact, 
\cite{BakerBessis1981} provides strong evidence it should not hold in general.

Nevertheless, the above strategy can still be used in the magnetised regime. The idea is to analyse a quantity which differs slightly from the truncated two-point function. In \cite{AizenmanFernandezCriticalBehaviorMagnetization1986}, the authors use the \emph{random current expansion} of the model, or more specifically the associated \emph{backbone expansion} (see Section \ref{sec:backbone and stuff}) to define, for every $\beta,h\geq 0$, a two-point function $K_{\beta,h}$ such that, for every $x\in \mathbb Z^d$,
\begin{equation}\label{eq:inequality K intrp}
	\langle \sigma_0;\sigma_x\rangle_{\beta,h}\leq K_{\beta,h}(0,x)\leq \langle \sigma_0\sigma_x\rangle_{\beta}.
\end{equation}
The upshot of this replacement is that a version of Lemma \ref{lem:SL} holds true for $K$. Moreover, we do not lose too much information with this bound as the truncated two-point function and $K$ are comparable on average, see Remark \ref{rem: K and truncated}. The definition and the main properties of $K$ are presented in Section \ref{sec:2pt K}; the Simon--Lieb-type inequality it satisfies is described in Section \ref{sec:SL for K}. From $K$, one may define a quantity $\Psi_{E^c,\beta,h}(S)$, parametrised by some $E\subset E(\mathbb Z^d)$, and whose precise definition can be found in Lemma \ref{lem: SL for K}. The quantity $\overline{\Psi}_{\beta,h}(S):=\sup_{E\subset E(\mathbb Z^d)}\Psi_{E^c,\beta,h}(S)$ is the natural analogue of $\varphi_{\beta}(S)$. The supremum is present here for technical reasons which are explained in Remark \ref{rem:comp psi phi} (in words, $K$ does not satisfy any monotonicity property in the volume). As stated in Proposition \ref{prop:sufficient condition to exp decay}, the existence of $S\ni 0$ finite such that $\overline{\Psi}_{\beta,h}(S)< \tfrac{1}{2}$ yields an upper bound on the correlation length of $K$---and therefore on $\xi(\beta,h)$ by \eqref{eq:inequality K intrp}---of order $\textup{diam}(S)$. The remaining task is therefore to find such a set $S$. Our approach is as follows: we start from \cite{DuminilPanis2024newLB} which gives such a result when $h=0$ and extend it to the regime $h>0$ via an 
\emph{extrapolation} argument inspired by \cite{AizenmanFernandezCriticalBehaviorMagnetization1986}. 

As we explain in more detail below, in the thermal regime $\ell(\beta,h)=(\beta_c-\beta)^{-1/2}$, the bound \eqref{eq:main upper intro} follows essentially from \cite{DuminilPanis2024newLB} (see Proposition \ref{prop:trunc vs K vs untrunc} and Theorem \ref{thm:sharplength}). The most challenging case is therefore the magnetic regime $\ell(\beta,h)=h^{-1/3}$. To simplify the exposition, we work at $(\beta_c,h)$ with $h>0$.

The ansatz~\eqref{eq:expected behaviour in the near-critical regime} suggests
that the truncated two-point function remains ``stable'' along level sets of
the correlation length $\xi(\beta,h)$ in $\mathcal Q$. It is reasonable to expect a similar behaviour for other 
observables of the model. In particular, $\overline{\Psi}_{\beta_c,h}(S)$ 
should be comparable to $\overline{\Psi}_{\beta(h),0}(S)$ for a 
well-chosen $\beta(h)<\beta_c$:
\begin{equation}\label{eq:expl comparison to h=0 0}
	\overline{\Psi}_{\beta_c,h}(S)\leq \overline{\Psi}_{\beta(h),0}(S)+\varepsilon,
\end{equation}
where $\varepsilon$ is a small \emph{error} term. If $\xi(\beta,h)\asymp \ell(\beta,h)$, the above heuristic forces us to choose $\beta(h)$ such that $\ell(\beta(h),0)\asymp\ell(\beta_c,h)$, which gives
\begin{equation}
    \beta(h)=\beta_c-ch^{2/3},
\end{equation}
for some $c>0$ to be fixed. For now, it is unclear how $c$ relates to $\varepsilon$ in \eqref{eq:expl comparison to h=0 0}. From the result of \cite{DuminilPanis2024newLB} described above, one can find a finite 
set $S\ni 0$ of diameter $\lesssim(\beta_c-\beta(h))^{-1/2}\asymp h^{-1/3}=\ell(\beta_c,h)$ 
with $\varphi_{\beta(h)}(S)<\tfrac{1}{2}$. As a result, if $\varepsilon$ is small enough,
\begin{equation}\label{eq:expl comparison to h=0 1}
    \overline{\Psi}_{\beta_c,h}(S)
    \leq\overline{\Psi}_{\beta(h),0}(S)+\varepsilon
    \leq\varphi_{\beta(h)}(S)+\varepsilon
    <\frac{1}{2},
\end{equation}
where the second inequality follows from the appropriate generalisation of \eqref{eq:inequality K intrp} (see Proposition \ref{prop:trunc vs K vs untrunc}). As claimed above (and proved in 
Proposition~\ref{prop:sufficient condition to exp decay}), this suffices 
to establish the upper bound in Theorem~\ref{thm:main}.

It remains to prove~\eqref{eq:expl comparison to h=0 1}. For this, we are inspired by \cite{AizenmanFernandezCriticalBehaviorMagnetization1986}. In that paper, Aizenman and Fern\'andez prove that, in dimensions $d>4$,
\begin{equation}
	M(\beta_c,h)=\langle \sigma_0\rangle_{\beta_c,h}\lesssim h^{1/3}.
\end{equation}
They then bound $m^*(\beta)$ (for $\beta>\beta_c$) by $M(\beta_c,h(\beta))$ for a well-chosen $h(\beta)>0$. More precisely, the authors of \cite{AizenmanFernandezCriticalBehaviorMagnetization1986} construct interpolations $\boldsymbol{\beta}:[0,1]\rightarrow \mathbb R^+$, with $\boldsymbol{\beta}(1)=\beta$ and $\boldsymbol{\beta}(0)=\beta_c$, and $\boldsymbol{h}:[0,1]\rightarrow \mathbb R^+$ with\footnote{By this, we 
mean that we set $\boldsymbol{h}(1)=\kappa$ and send $\kappa$ to $0$.}  $\boldsymbol{h}(1)=0^+$ and $\boldsymbol{h}(0)=h(\beta)$ such that
\begin{equation}\label{eq:expl comparison to h=0 2}
	\frac{\partial}{\partial t}M(\boldsymbol{\beta}(t),\boldsymbol{h}(t))\leq 0,
\end{equation} 
thereby proving that 
$M(\beta,0^+)=m^*(\beta)\leq M(\beta_c,h(\beta))$. The key input in the proof of \eqref{eq:expl comparison to h=0 2} is the following inequality (which is a consequence of \cite{GriffithsHurstShermanConcavity1970}) relating $\beta$- and $h$-derivatives of $M(\beta,h)$:
\begin{equation}\label{eq:expl comparison to h=0 2.5}
	\frac{\partial}{\partial \beta}M(\beta,h)\leq 2d M(\beta,h)\frac{\partial}{\partial h}M(\beta,h).
\end{equation}

It is natural to use a similar argument to compare $\overline{\Psi}_{\beta_c,h}(S)$ to 
$\overline{\Psi}_{\beta(h),0}(S)\leq\varphi_{\beta(h)}(S)$. However, for this more intricate quantity, we lack an analogue of \eqref{eq:expl comparison to h=0 2.5}, which makes it much harder to find an interpolation from $(\beta(h),0)$ to $(\beta_c,h)$ along which the derivative remains negative, as in \eqref{eq:expl comparison to h=0 2}. Nevertheless, we can still proceed as follows. Setting 
$\boldsymbol{\beta}(t):=\beta(h)+t(\beta_c-\beta(h))$ and applying the 
fundamental theorem of calculus gives
\begin{equation}\label{eq:expl comparison to h=0 3}
    \overline{\Psi}_{\beta_c,h}(S)
    =\overline{\Psi}_{\beta(h),h}(S)
    +\int_0^1\frac{\partial}{\partial t}
    \overline{\Psi}_{\boldsymbol{\beta}(t),h}(S)\,\mathrm{d}t
    \leq\varphi_{\beta(h)}(S)
    +\int_0^1\frac{\partial}{\partial t}
    \overline{\Psi}_{\boldsymbol{\beta}(t),h}(S)\,\mathrm{d}t,
\end{equation}
where we omitted the technical difficulty coming from the potential non-differentiability of $\overline{\Psi}$. The 
integral in~\eqref{eq:expl comparison to h=0 3} is then analysed using a (new) bound on the $\beta$-derivative of $K$, see Proposition \ref{prop: derivative K}. We leave the technical details for later and record only that the derivative can be bounded by
\begin{equation}
    C_1(d)(\beta_c-\beta(h))\sup_{E\subset E(\mathbb{Z}^d)}
    \sup_{t\in[0,1]}\sum_{u\in S}
    K_{E^c,\boldsymbol{\beta}(t),h}(0,u)\,\varphi_{\beta_c}(S-u).
\end{equation}
As in the case of Bernoulli percolation in dimensions $d>6$ (see 
\cite{panis2026reversing}), one expects that there exists $C_2(d)>0$ such that
\begin{equation}\label{eq:expl comparison to h=0 4}
	\max_{u\in S}\varphi_{\beta_c}(S-u)\leq C_2(d).
\end{equation}
Assuming such a bound holds, we can again rely on \cite{AizenmanFernandezCriticalBehaviorMagnetization1986} (see Proposition \ref{prop: sum of K}) to obtain that
\begin{equation}
	\sup_{E\subset E(\mathbb{Z}^d)}
    \sup_{t\in[0,1]}\sum_{u\in S}
    K_{E^c,\boldsymbol{\beta}(t),h}(0,u)\lesssim h^{-1}M(\beta_c,h)\lesssim h^{-2/3}.
\end{equation}
Therefore, the integral 
in~\eqref{eq:expl comparison to h=0 3} is bounded by
\begin{equation}
    C_3(d)(\beta_c-\beta(h))h^{-2/3}\leq C_3(d) c,
\end{equation}
which can be made arbitrarily small by choosing $c$ small 
enough.

There is, however, one fundamental caveat in the above argument: \eqref{eq:expl comparison to h=0 4} is not known for the Ising model.
We circumvent 
this difficulty by replacing the deterministic set $S$ by a well-chosen \emph{random} set 
$\mathcal{S}$ for which the relevant estimates---in particular \eqref{eq:expl comparison to h=0 4}---hold in an 
averaged sense (see Corollary \ref{coro:bound translate random s_n}). This motivates the generalisation of Proposition \ref{prop:sufficient condition to exp decay} presented in 
Proposition~\ref{prop:sufficient condition to exp decay random}. The 
random set $\mathcal{S}$ is the same as the one considered in \cite{DuminilPanis2024newLB}. It is constructed from a \emph{folded} single 
sourceless random current on $\mathbb{Z}^d$, following an approach 
introduced in \cite{AizenmanDuminilTassionWarzelEmergentPlanarity2019}.

\begin{Rem}
It may seem surprising to change only $\beta$ and not $h$. One might hope 
that decreasing $h$ as well could yield a better bound on the integral 
in~\eqref{eq:expl comparison to h=0 3}. This is not necessary: since 
$\ell(\beta(h),h)\asymp\ell(\beta(h),0)$, the field $h$ is not expected 
to affect the mass beyond a multiplicative constant. Removing $h$ 
via~\eqref{eq:inequality K intrp} is therefore harmless for our purposes.
\end{Rem}
 

Finally, the proof of \eqref{eq:main lower intro} uses a differentiation argument which leverages estimates from \cite{AizenmanFernandezCriticalBehaviorMagnetization1986,DuminilPanis2024newLB}. By \cite{AizenmanFernandezCriticalBehaviorMagnetization1986}, it is possible to bound $|\partial_h\langle \sigma_0;\sigma_x\rangle_{\beta,h}|$ in terms of a diagram involving the two-point function $K$ and the magnetisation $M(\beta,h)$, see \eqref{eq: lb on derivative in h of 2pt}. The latter quantity is then controlled using known critical estimates, which we recall in Section \ref{sec:preliminaries}. 

\hfill

\noindent\textbf{Organisation of the paper.} In Section \ref{sec:preliminaries}, we recall useful results regarding the Ising model. The central tool of this paper---the two-point function $K$---is introduced in Section \ref{sec:backbone and stuff}, after recalling the random current and backbone expansions of the Ising model. In Section \ref{sec:SL for K}, we prove a version of the Simon--Lieb inequality for $K$. In Section \ref{sec:extrapolation}, we implement the extrapolation argument described above. In particular, we introduce a random set $\mathcal S_n\subset \Lambda_n$ which satisfies a set of ``good'' properties, see Proposition \ref{prop:properties of the random set}. The proof of the latter result is completed in Section \ref{sec: properties of random set}. Finally, in Section \ref{sec:lower bounds}, we prove the lower bounds \eqref{eq:main lower intro} and \eqref{eq:main lower four intro}.

\hfill

\noindent \textbf{Acknowledgements.} We thank Frederik Ravn Klausen, S\'ebastien Ott, Franco Severo, and Peter Wildemann for stimulating discussions. We also thank Gordon Slade for useful comments leading to the scaling theory discussion above. We acknowledge support from the Swiss National Science Foundation through a Postdoc.Mobility grant.

\hfill

\noindent\textbf{Statement of AI use.} We have used a combination ChatGPT~5.6 Plus and Claude Sonnet 4.6 to assist with the English writing and to identify minor 
errors and typos. They have not been used to generate ideas or proofs, or to 
suggest references.

\section{Preliminaries}\label{sec:preliminaries}
In this section, we recall useful results regarding the high-dimensional Ising model. Our main inputs come from \cite{AizenmanFernandezCriticalBehaviorMagnetization1986} and \cite{DuminilPanis2024newLB}. Recall that $\mathcal Q=\{(\beta,h): 0\leq\beta\leq \beta_c,\: h\geq 0\}$, and that for $(\beta,h)\in \mathcal Q^*$, 
$\ell(\beta,h)=(\beta_c-\beta)^{-1/2}\wedge h^{-1/3}=\ell(\beta,0)\wedge \ell(\beta_c,h)$.

\subsection{Correlation inequalities}

We begin by stating classical correlation inequalities. For a set $A\subset \mathbb Z^d$, we let 
\begin{equation}
\sigma_A:=\prod_{x\in A}\sigma_x.
\end{equation} 

\begin{Prop}[\hspace{1pt}{\cite{GriffithsCorrelationsIsing1-1967}}]\label{prop:secondineqGriffiths} Let $d\geq 2$ and $G=(V,E)$ be a subgraph of $\mathbb Z^d$. For every $J\in (\mathbb R^+)^{E}$, every $h\geq 0$, and every $A,B\subset V$,
\begin{equation}
	\langle \sigma_A;\sigma_B\rangle_{G,J,h}:=\langle \sigma_A\sigma_B\rangle_{G,J,h}-\langle \sigma_A\rangle_{G,J,h}\langle \sigma_B\rangle_{G,J,h}
\geq 0.
\end{equation}
\end{Prop}

We now state a fundamental monotonicity result. Its proof can be found in \cite[Chapter~3]{FriedliVelenikIntroStatMech2017}. Later, we will refer to it as \emph{Griffiths' inequality}.

\begin{Prop}[\hspace{1pt}{\cite{GriffithsCorrelationsIsing1-1967,kelly1968general}}]\label{prop:consequencegriffiths} Let $d\geq 2$ and $G=(V,E)$, $H=(V',E')$ be two subgraphs of $\mathbb Z^d$ such that $G\subset H$. For every $A\subset V$, every $J\in (\mathbb R^+)^E$, $J'\in (\mathbb R^+)^{E'}$ satisfying $J_e\leq J_e'$ for every $e\in E$, and every $h,h'\geq 0$ satisfying $h\leq h'$, one has
\begin{equation}\label{eq:monot Griffiths}
	\langle \sigma_A\rangle_{G,J,h}\leq \langle \sigma_A\rangle_{H,J',h'}.
\end{equation}
\end{Prop}
We conclude with a useful bound on the truncated two-point function. We provide a more precise result in Proposition \ref{prop:trunc vs K vs untrunc} below. A stronger inequality, allowing $h$ to take negative values, has recently been derived in \cite{Ding2023new}.
\begin{Prop}\label{prop: trunc is bounded by h=0} Let $d\geq 2$ and $G=(V,E)$ be a subgraph of $\mathbb Z^d$. For every $J\in (\mathbb R^+)^E$, every $h\geq 0$, and every $x,y\in V$, one has
\begin{equation}
	\langle \sigma_x;\sigma_y\rangle_{G,J,h}\leq \langle \sigma_x\sigma_y\rangle_{G,J}.
\end{equation}
\end{Prop}

\subsection{Two-point function and sharp length estimates}

We now turn to results regarding the model's two-point function at $h=0$. It is a classical consequence of the infrared bound and the Messager--Miracle-Sol\'e inequalities \cite{FrohlichSimonSpencerIRBounds1976,FrohlichIsraelLiebSimon1978,MessagerMiracleSoleInequalityIsing} that, for every $d\geq 3$, there exists $C>0$ such that, for every $x\in \mathbb Z^d$,
\begin{equation}\label{eq:IRB}
	\langle \sigma_0\sigma_x\rangle_{\beta_c}\leq \frac{C}{1\vee |x|^{d-2}}.
\end{equation}
In dimensions $d\geq 4$, this upper bound is expected to be \emph{sharp}, in the sense that a matching order lower bound should hold, see \cite{SladeTombergRGWeakPhi4in2016}. This result was recently obtained in dimensions $d>4$ in \cite{DuminilPanis2024newLB} (see also \cite{Sakai2007LaceExpIsing} for $d\gg 4$). In the same paper, the authors obtained \emph{near-critical bounds} on the two-point function $\langle \sigma_0\sigma_x\rangle_{\beta}$. These results extend to the marginal dimension $d=4$ up to appropriate logarithmic corrections. 

\begin{Thm}[\hspace{1pt}{\cite{DuminilPanis2024newLB}}]\label{thm:bounds 2pt d>4} Let $d>4$. There exist $c,C>0$ such that the following holds. For every $\beta\in [0,\beta_c]$ and every $x\in \mathbb Z^d$,
\begin{equation}
	\langle \sigma_0\sigma_x\rangle_{\beta}\leq \frac{C}{1\vee |x|^{d-2}}\exp\left(-c\frac{|x|}{\ell(\beta,0)}\right).
\end{equation}
Moreover, if $|x|\leq c\,\ell(\beta,0)$,
\begin{equation}
	\langle \sigma_0\sigma_x\rangle_{\beta}\geq \frac{c}{1\vee |x|^{d-2}}.
\end{equation}
\end{Thm}

\begin{Thm}[\hspace{1pt}{\cite{DuminilPanis2024newLB}}]\label{thm:bounds 2pt d=4} Let $d=4$. There exist $c,C>0$ such that the following holds. For every $\beta\in [0,\beta_c]$ and every $x\in \mathbb Z^4$,
\begin{equation}\label{eq:nc upper bound d=4 h=0}
	\langle \sigma_0\sigma_x\rangle_{\beta}\leq \frac{C}{1\vee |x|^{2}}\exp\left(-c\frac{|x|}{\ell(\beta,0)\log \ell(\beta,0)}\right).
\end{equation}
Moreover, for every $|x|\leq c\,\ell(\beta,0)$,
\begin{equation}
	\langle \sigma_0\sigma_x\rangle_{\beta}\geq \frac{c}{1\vee |x|^2\log |x|}.
\end{equation}
\end{Thm}

As stated in \cite[Theorem~1.6]{DuminilPanis2024newLB}, these results imply that the correlation length of the Ising model at parameters $(\beta,0)$ with $\beta<\beta_c$ is of order $\ell(\beta,0)^{1+o(1)}$, where $o(1)$ tends to $0$ as $\beta$ approaches $\beta_c$. However, the main quantity of interest in \cite{DuminilPanis2024newLB} is not the correlation length, but rather the so-called \emph{sharp length} (introduced in \cite{hutchcroft2022derivation,PanisTriviality2023}), which we now define.

Recall the definition of $\varphi_{\beta}(S)$ from \eqref{eq: def phi}. The sharp length $L(\beta)$ is defined by 
\begin{equation}\label{eq:def L(beta)}
    L(\beta)=\inf\left\lbrace k\geq 1: \textup{ There exists } S\subset \Lambda_k \textup{ with }S\ni 0 \textup{ such that } \varphi_{\beta}(S)< \tfrac{1}{4}\right\rbrace.
\end{equation}
The following result\footnote{To be more precise, \cite{DuminilPanis2024newLB} uses $1/2$ instead of $1/4$ in the definition of $L(\beta)$. However, this minor modification does not affect their results.} was obtained in \cite{DuminilPanis2024newLB}, see Theorem~1.6 and Remark~1.7 there.
\begin{Thm}[{\hspace{1pt}{\cite{DuminilPanis2024newLB}}}]\label{thm:sharplength} Let $d\geq 4$. There exist $c,C>0$ such that, for every $\beta\in (0,\beta_c)$,
\begin{equation}
	c\,\ell(\beta,0)\leq L(\beta)\leq C\left\{
			\begin{array}{ll}
				\ell(\beta,0) & \mbox{if } d>4, \\
				\ell(\beta,0)\log \ell(\beta,0) & \mbox{if }d=4.
			\end{array}
			\right.
\end{equation}
\end{Thm}
%
%

\subsection{Behaviour of the magnetisation}
In the rest of the section, we recall results from \cite{AizenmanFernandezCriticalBehaviorMagnetization1986} on the high-dimensional Ising model at parameters $(\beta,h)\in \mathcal Q$. Recall that
\begin{equation}
	M(\beta,h)=\langle \sigma_0\rangle_{\beta,h}.
\end{equation}

\begin{Thm}[\hspace{1pt}{\cite{AizenmanFernandezCriticalBehaviorMagnetization1986}}]\label{thm: magnetisation} Let $d\geq 4$. There exist $c,C,h_0>0$ such that, for every $h\in(0,h_0]$,
\begin{equation}
	c\,h^{1/3}\leq M(\beta_c,h)\leq C\left\{
			\begin{array}{ll}
				h^{1/3} & \mbox{if } d>4, \\
				h^{1/3}|\log h| & \mbox{if }d=4.
			\end{array}
			\right.
	\end{equation}
\end{Thm}
\begin{Rem} By Griffiths' inequality, the upper bound in Theorem \ref{thm: magnetisation} holds with $\beta_c$ replaced by $\beta\leq \beta_c$.
\end{Rem}

\subsection{Behaviour of the susceptibility}
We end this section with bounds on the truncated susceptibility $\chi(\beta,h)$. Recall that it is defined for $(\beta,h)\in \mathcal Q^*$ by
\begin{equation}
	\chi(\beta,h)=\sum_{x\in \mathbb Z^d}\langle \sigma_0;\sigma_x\rangle_{\beta,h}.
\end{equation}
The following result is an easy consequence of the analysis of \cite{AizenmanGeometricAnalysis1982,AizenmanGrahamRenormalizedCouplingSusceptibility4d1983,AizenmanFernandezCriticalBehaviorMagnetization1986}. We provide a short proof below for the sake of completeness.
\begin{Thm}\label{thm: susceptibility} Let $d\geq 4$. There exist $C,\delta>0$ such that, for every $(\beta,h)\in \mathcal Q_\delta$,
\begin{equation}
\chi(\beta,h)\leq C\left\{
			\begin{array}{ll}
				\ell(\beta,h)^2 & \mbox{if } d>4, \\
				\ell(\beta,h)^2\log \ell(\beta,h) & \mbox{if }d=4.
			\end{array}
			\right.
\end{equation}
\end{Thm}

\begin{proof} By definition of $\ell(\beta,h)$, it suffices to prove that 
\begin{equation}
\chi(\beta,h)\lesssim \ell(\beta,0)^2(\log \ell(\beta,0))^{\mathds{1}_{d=4}},
\end{equation} and 
\begin{equation}
\chi(\beta,h)\lesssim \ell(\beta_c,h)^2(\log \ell(\beta_c,h))^{\mathds{1}_{d=4}}.
\end{equation} 

The first inequality follows by the results of \cite{AizenmanGeometricAnalysis1982} (in $d>4$) and \cite{AizenmanGrahamRenormalizedCouplingSusceptibility4d1983} (in $d=4$) after observing that, thanks to Proposition \ref{prop: trunc is bounded by h=0}, one has $\chi(\beta,h)\leq \chi(\beta,0)$. 

To prove the second inequality, we use Proposition \ref{prop: trunc is bounded by h=0} and Proposition \ref{prop: sum of K} below (which is Proposition 4.8 in \cite{AizenmanFernandezCriticalBehaviorMagnetization1986}) to get that $\chi(\beta,h)\lesssim M(\beta,h) h^{-1}$. Combining this with Theorem \ref{thm: magnetisation} concludes the proof. 
%
%
%
\end{proof}

\begin{Rem}\label{rem:suscept lb}\hspace{-1pt}(1) A non-rigorous heuristic for Theorem \ref{thm: susceptibility} (at least when $\beta=\beta_c$) is provided by the combination of Theorem \ref{thm: magnetisation} and the following identity
\begin{equation}
	\chi(\beta,h)=\frac{\partial}{\partial h}M(\beta,h).
\end{equation}

(2) In fact, \cite{AizenmanFernandezCriticalBehaviorMagnetization1986} 
establishes a matching lower bound of the same order, at least when 
$\ell(\beta,h)=\ell(\beta_c,h)$. In particular, it holds that $\chi(\beta_c,h)\gtrsim \ell(\beta_c,h)^2$ when $d>4$.
\end{Rem}

\section{Backbone expansion and the two-point function $K$}\label{sec:backbone and stuff}

We now recall the \emph{backbone expansion} (or random walk expansion) of the Ising model from which the two-point function $K$---the central object in our proofs---emerges naturally. The backbone expansion is closely related to the \emph{random current expansion} of the model, which we first introduce.
\subsection{Random currents}
%
The random current expansion of the Ising model goes back to \cite{GriffithsHurstShermanConcavity1970}. Its key role in the analysis of the high-dimensional Ising model was first observed in \cite{AizenmanGeometricAnalysis1982}.  Since then, it has been the central object of many works, see for instance \cite{AizenmanGrahamRenormalizedCouplingSusceptibility4d1983,AizenmanFernandezCriticalBehaviorMagnetization1986,AizenmanBarskyFernandezSharpnessIsing1987,AizenmanDuminilSidoraviciusContinuityIsing2015,DuminilTassionNewProofSharpness2016,AizenmanDuminilTassionWarzelEmergentPlanarity2019,AizenmanDuminilTriviality2021,PanisTriviality2023,KrachunPanagiotisPanisScalinglimit2023,chen2023conformal,PanisIIC2024,DuminilPanis2024newLB,duminil2025conformal,LiuPanisSladePlateau,thinggaard2026end,hansen2026supercritical}. We refer to \cite{DuminilLecturesOnIsingandPottsModels2019} or \cite[Chapter~2]{Panis2024applications} for more information.

Let $G=(V,E)$ be a finite subgraph of $\mathbb Z^d$. Let $\mathfrak{g}$ denote the \emph{ghost} vertex and define $G_{\mathfrak{g}}$ to be the graph of vertex set $V_\mathfrak{g}:=V\cup \{\mathfrak{g}\}$ and edge set $E_{\mathfrak{g}}:=E\cup \{x\mathfrak{g}: x\in V\}$.
\begin{Def}[Current] A \textit{current} $\n$ on $G_\fg$ is an $\mathbb N$-valued function on $E_\fg$, where $\mathbb N=\lbrace 0,1,\ldots\rbrace$. We let $\Omega_{G_\fg}$ denote the set of currents on $G_\fg$. The set of \emph{sources} of $\n$, denoted by $\sn$, is defined as follows
\begin{equation}
\sn:=\Big\{ x \in V_\fg \: : \: \sum_{y\in V_\fg:\: xy\in E_\fg}\n_{xy}\textup{ is odd}\Big\}.
\end{equation}
For every $J\in (\mathbb R^+)^E$ and $h\geq 0$, we let
\begin{equation}
w_{J,h}^G(\n):=\prod_{\substack{xy\in E}}\dfrac{J_{xy}^{\n_{xy}}}{\n_{xy}!}\prod_{x\in V}\frac{h^{\n_{x\fg}}}{\n_{x\fg}!}.
\end{equation}
If $J\equiv \beta$ for some $\beta\geq 0$ (resp. $h=0$) we simply write $w_{\beta,h}^G(\n)=w_{J,h}^G(\n)$ (resp. $w_{J}^G(\n)=w_{J,0}^G(\n)$). 
\end{Def}
\begin{Rem} Observe that $|\sn|$ is always even.
\end{Rem}

It is classical (see \cite{AizenmanGeometricAnalysis1982,DuminilLecturesOnIsingandPottsModels2019,Panis2024applications}) that the correlation functions of the Ising model are related to currents: setting $\sigma_{\fg}=1$, one has, for every $A\subset V_\fg$ with $|A|$ even, 
\begin{equation}\label{eq: relation partition functions}
	\sum_{\sigma\in \{-1,1\}^V}\sigma_A \exp\Big(\sum_{xy\in E}J_{xy}\sigma_x\sigma_y+h\sum_{x\in V}\sigma_x\Big)=2^{|V|}\sum_{\n\in \Omega_{G_\fg}:\: \sn=A}w^G_{J,h}(\n),
\end{equation}
so that
\begin{equation}\label{equation correlation rcr}
\left\langle \sigma_A\right\rangle_{G,J,h}=\dfrac{\sum_{\n\in \Omega_{G_\fg}:\:\sn=A}w_{J,h}^G(\n)}{\sum_{\n\in \Omega_{G_\fg}:\:\sn=\emptyset}w_{J,h}^G(\n)}. 
\end{equation}
If $A\subset V$ with $|A|$ even, we define a probability measure $\mathbf{P}_{G,J,h}^A$ on $\Omega_{G_\fg}$ as follows: for every $\n\in \Omega_{G_\fg}$, 
\begin{equation}
    \mathbf{P}_{G,J,h}^A[\n]:=\mathds{1}_{\sn=A}\frac{w_{J,h}^G(\n)}{Z_{G,J,h}^A},
\end{equation}
where $Z_{G,J,h}^A:=\sum_{\n\in \Omega_{G_\fg}:\:\sn=A}w_{J,h}^G(\n)$ is a normalisation constant. When $A=\lbrace x,y\rbrace$, we write $\mathbf{P}^{xy}_{G,J,h}$ instead of $\mathbf{P}^{\{x,y\}}_{G,J,h}$. As observed in \cite{AizenmanDuminilSidoraviciusContinuityIsing2015}, one can construct the infinite volume limit of the above measures as $G\nearrow \mathbb Z^d$. We denote it by $\mathbf P_{J,h}^A$. When $J\equiv\beta$ for some $\beta\geq 0$, or when $h=0$, we use the simplified notation introduced above.

The trace $(\mathds{1}_{\n_{xy}>0})_{xy \in E_\fg}$ of a current $\n$ induces a percolation configuration on $G_\fg$.   The main strength of the random current expansion lies in a combinatorial identity called the \emph{switching lemma}, which relates the percolation induced by the sum of two independent currents to the Ising model. We will not directly use it here but refer the reader to \cite{DuminilLecturesOnIsingandPottsModels2019,Panis2024applications} for more information. The switching lemma provides a probabilistic interpretation of the truncated two-point function $\langle \sigma_x;\sigma_y\rangle_{G,\beta,h}$, which we now state to motivate what is below: for every $\beta,h\geq 0$ and $x,y\in V$ distinct, one has
\begin{equation}\label{eq:switching for truncated}
	\langle \sigma_x;\sigma_y\rangle_{G,\beta,h}=\langle \sigma_x\sigma_y\rangle_{G,\beta,h}\mathbf P^{xy,\emptyset}_{G,\beta,h}[x \textup{ is not connected to }\fg\textup{ in }\n_1+\n_2],
\end{equation}
where $\mathbf P^{xy,\emptyset}_{G,\beta,h}$ denotes the product measure $\mathbf P^{xy}_{G,\beta,h}\otimes\mathbf P^{\emptyset}_{G,\beta,h}$. In words, the probabilistic interpretation of the truncated two-point function is that, in the appropriate double random current percolation, $x$ and $y$ are connected together but not to $\fg$.

\subsection{Backbone expansion}
If $\n$ is a current with source set $\sn=\{x,y\}$, then $x$ and $y$ are connected in $\n$. The backbone expansion of the Ising model---first introduced in \cite{AizenmanGeometricAnalysis1982}---provides a way to explore one such path.
We follow the definitions of \cite{AizenmanFernandezCriticalBehaviorMagnetization1986} (except for one minor difference), and refer to this paper for more information.

We fix $G=(V,E)$ a finite subgraph of $\mathbb Z^d$ and an arbitrary ordering $\prec$ of $E$. We construct an ordering of $E_\fg$ from $\prec$---that we still denote $\prec$---by requiring that for each $x,y\in V$ with $y\sim x$, one has $xy\prec x\fg$. 

\begin{Rem}
The ordering $\prec$ on $E_\fg$ differs from that of 
\cite{AizenmanFernandezCriticalBehaviorMagnetization1986}, where the edge 
from any vertex $x$ to the ghost is always the earliest. This difference 
does not affect the results of 
\cite{AizenmanFernandezCriticalBehaviorMagnetization1986} that we rely on (except for the lower bound in \eqref{eq: prop sum K 1} below), 
and is motivated by technical reasons explained below.
\end{Rem}

 To each edge (or step) $xy$ (with $y\in V_\fg$), we introduce its set of \emph{explored edges}, defined to be $\{xz: xz\prec xy\}\cup \{xy\}$. A sequence of steps $\gamma=(x_ix_{i+1})_{0\leq i\leq n-1}$ is said to be \emph{consistent} if no step uses an edge explored by a previous step. We denote by $\overline{\gamma}$ the set of all edges explored by $\gamma$, and by $\partial \gamma$ the vertices of $V_\fg$ covered by an odd number of steps of $\gamma$. 

We are now equipped to describe the backbone exploration.

\begin{Def}[Backbone exploration from a vertex] Let $x\in V_\fg$. Let $\n\in \Omega_{G_{\fg}}$ be such that $x\in \sn$. The $x$-backbone of $\n$, denoted $\Gamma_x(\n)$, is constructed as follows:
\begin{enumerate}
	\item[$(i)$] Set $x_0=x$. Explore from $x$ the earliest edge $e$ such that $\n_e$ is odd. Call $x_1$ its endpoint.
	 \item[$(ii)$] Assume that $(x_i)_{0\leq i\leq k}$ has been constructed and that $x_k\notin(\sn\setminus\{x\})\cup\{\fg\}$. Explore from $x_k$ the earliest edge $e$ which has not been explored so far and such that $\n_e$ is odd.
	 \item[$(iii)$] The exploration necessarily stops at an element of $(\sn\setminus\{x\})\cup\{\fg\}$ after $n$ steps, for some $n\geq 1$. We let $\Gamma_x(\n)$ denote the sequence of $n$ steps $(x_ix_{i+1})_{0\leq i \leq n-1}$.
		\end{enumerate}
		\end{Def}

	\begin{Def}[Backbone exploration] Let $x,y\in V_\fg$. Let $\n\in \Omega_{G_\fg}$ be such that $\sn=\{x,y\}$. The backbone of $\n$, denoted $\Gamma(\n)$, is constructed as follows:
	\begin{enumerate}
		\item[$(i)$] Explore the $x$-backbone $\Gamma_x(\n)$. If it ends at $y$, set $\Gamma(\n):=\Gamma_x(\n)$.
		\item[$(ii)$] If not, run the exploration of the $y$-backbone in the complement of $\overline{\Gamma_x(\n)}$. The backbone $\Gamma(\n)$ is then defined to be the union of these two explorations.
	\end{enumerate}
	\end{Def}

\begin{Rem}\label{rem:writing with paths}
In the first case above, the backbone $\Gamma(\n)$ can be seen as a path $\gamma:x\rightarrow y$. In the second case, it is a pair of paths $(\gamma_1,\gamma_2)$ with $\gamma_1:x\rightarrow \fg$ and $\gamma_2:y\rightarrow \fg$, and where $\gamma_2$ takes steps in $E_\fg\setminus \overline{\gamma_1}$.
\end{Rem}	
Let us make a few observations. By definition, the set $\overline{\Gamma(\n)}$ is determined by $\Gamma(\n)$ and $\prec$. Similarly, $\Gamma(\n)$ can be deduced from $\overline{\Gamma(\n)}$ by choosing the odd edges according to $\prec$. The current $\n\setminus\overline{\Gamma(\n)}$, which we define to be the restriction of $\n$ to the complement of $\overline{\Gamma(\n)}$, is sourceless. If $J\in (\mathbb R^+)^E$ and $h\geq 0$, we can write
\begin{equation}\label{eq:backbone expansion}
	\langle \sigma_x\sigma_y\rangle_{G,J,h}=\sum_{\partial \gamma=\{x,y\}}\rho_{G,J,h}(\gamma),
\end{equation}
where, for a sequence of steps $\gamma$,\begin{equation}
	\rho_{G,J,h}(\gamma):=\langle \sigma_{\partial \gamma}\rangle_{G,J,h}\mathbf P^{\partial \gamma}_{G,J,h}[\Gamma(\n)=\gamma].
\end{equation}
Observe that, by definition, the sum in \eqref{eq:backbone expansion} is supported on sequences of steps that are consistent. 
When $J\equiv \beta$ for some $\beta\geq 0$ (resp. $h=0$), we write $\rho_{G,\beta,h}(\gamma)=\rho_{G,J,h}(\gamma)$ (resp. $\rho_{G,J}(\gamma)=\rho_{G,J,0}(\gamma)$). Additionally, if $S\subset V_\fg$, we let 
\begin{equation}\label{eq:def rho S}
\rho_{G,J,h}^S(\gamma)=\mathds{1}_{\gamma\subset S}\rho_{G,J,h}(\gamma).
\end{equation}
 Again, \cite{AizenmanDuminilSidoraviciusContinuityIsing2015} allows to define the infinite volume version of the above quantities as $G\nearrow \mathbb Z^d$. We let $\rho_{J,h}=\rho_{\mathbb Z^d,J,h}$.
 
 \vspace{5pt}
 \noindent \textbf{Notation convention.} Before turning to properties of the backbone expansion, we fix a convention used 
throughout. It will be common to explore a path $\gamma\subset G_\fg$ and then run 
a second backbone exploration in the subgraph $G_\fg\setminus\overline{\gamma} 
:= (V_\fg, E_\fg\setminus\overline{\gamma})$. We abuse notation and write 
$\langle\cdot\rangle_{G\setminus\overline{\gamma},J,h}$ for the Ising measure 
on this subgraph (where we recall that $\sigma_\fg=1$), $Z^\emptyset_{G\setminus\overline{\gamma},J,h}$ for the 
associated sourceless random current partition function, and 
$\rho_{G\setminus\overline{\gamma},J,h}$ for the backbone expansion weights. 
 \vspace{5pt}

We now collect various useful properties of the backbone expansion. For proofs, we refer to \cite[Proposition~4.4]{AizenmanFernandezCriticalBehaviorMagnetization1986}. 

\begin{Prop}The following properties hold.
\label{prop:backbone}
\begin{enumerate}
	\item[$\bullet$] If $\gamma$ is a consistent sequence of steps and $\mathsf E$ is a subset of $E_\fg$ satisfying $\gamma\cap \mathsf E=\emptyset$, then
	\begin{equation}\label{eq:propbackb1}
		\rho_{G,J,h}(\gamma)\leq \rho_{G\setminus\mathsf E,J,h}(\gamma).
	\end{equation}
	\item[$\bullet$] If $\gamma$ is a consistent sequence of steps that is the concatenation of two consistent sequences of steps $\gamma_1$ and $\gamma_2$ (we write $\gamma=\gamma_1\circ \gamma_2)$, then
	\begin{equation}\label{eq:propbackb2}
		\rho_{G,J,h}(\gamma)=\rho_{G,J,h}(\gamma_1)\rho_{G\setminus \overline{\gamma_1},J,h}(\gamma_2).
	\end{equation}
	\item[$\bullet$] If $\gamma$ is a consistent sequence of steps, then
	\begin{align}\label{eq:formula backbone weight}
	\begin{aligned}	
		\rho_{G,J,h}(\gamma)&=\prod_{e\in \gamma}\tanh(J_e)\cdot\prod_{e\in \overline{\gamma}}\cosh(J_e)\cdot\frac{Z_{G\setminus \overline{\gamma},J,h}^\emptyset}{Z_{G,J,h}^\emptyset}\\&=\prod_{e\in \gamma}\tanh(J_e)\cdot\mathbf P^{\emptyset}_{G,J,h}[\n \textup{ is even on }\overline{\gamma}],
		\end{aligned}
	\end{align}
	where we have set $J_e=h$ on edges of the form $e=x\fg$ with $x\in V$. In particular, 
	\begin{equation}
	\rho_{G,\beta,h}(\gamma)\leq \prod_{e\in \gamma}\tanh(J_e).
	\end{equation} 
\end{enumerate}
\end{Prop}

\subsection{The two-point function $K$}\label{sec:2pt K}

We now introduce the two-point function $K$ described in the introduction, and state some properties of this object (old and new).

 Let $G=(V,E)$ be a subgraph of $\mathbb Z^d$ and fix $J\in (\mathbb R^+)^E$ and
$h\geq 0$. Let $x,y\in V$. By the backbone expansion introduced above (see Remark \ref{rem:writing with paths}),
\begin{equation}\label{eq:intro K 1}
	\langle \sigma_x\sigma_y\rangle_{G,J,h}=\sum_{\partial \gamma=\{x,y\}}\rho_{G,J,h}(\gamma)=\sum_{\gamma:x\rightarrow y}\rho_{G,J,h}(\gamma)+\sum_{\substack{\gamma_1:x\rightarrow \fg\\\gamma_2:y\rightarrow \fg}}\rho_{G,J,h}(\gamma_1\circ \gamma_2).
\end{equation}
The first sum in the right-hand side of \eqref{eq:intro K 1} measures contributions to $\langle \sigma_x\sigma_y\rangle_{G,J,h}$ coming from $\gamma$ not visiting $\fg$. It is reminiscent of what appears on the right-hand side of \eqref{eq:switching for truncated}. This motivates the following definition. If $S\subset V$ and $x,y\in V$, we define
\begin{equation}
	K_{G,J,h}^S(x,y):=\sum_{\gamma:x\rightarrow y}\rho_{G,J,h}^S(\gamma),
\end{equation}
with the convention $K_{G,J,h}^S(x,x)=\mathds{1}_{x\in S}$,
where we recall the definition of $\rho_{G,J,h}^S$ from \eqref{eq:def rho S}.
If $S=V$ we drop it from the notation. We also write $K_{J,h}=K_{\mathbb Z^d,J,h}$. If $J\equiv \beta$ for some $\beta\geq 0$ (resp. $h=0$) we write $K^S_{G,\beta,h}=K_{G,J,h}^S$ (resp. $K^S_{G,J}=K^S_{G,J,0}$). By \eqref{eq:propbackb1}, one has
\begin{equation}\label{eq: K^S vs K_S}
	K^S_{G,J,h}(x,y)\leq K_{S,J,h}(x,y),
\end{equation}
where we abused notation by viewing $S$ as the graph $(S,E(S))$.

We collect several useful properties of $K$. The next two results were derived in \cite{AizenmanFernandezCriticalBehaviorMagnetization1986}. We include their short proofs for the sake of completeness. The first property relates the two-point function $K$ to observables of the Ising model.
\begin{Prop}[{\hspace{1pt}\cite[Proposition~4.7]{AizenmanFernandezCriticalBehaviorMagnetization1986}}]\label{prop:trunc vs K vs untrunc} Let $G=(V,E)$ be a subgraph of $\mathbb Z^d$. For every $J\in (\mathbb R^+)^E$, every $h\geq 0$, and every $x,y\in V$,\begin{equation}
	\langle \sigma_x;\sigma_y\rangle_{G,J,h}\leq K_{G,J,h}(x,y)\leq \langle \sigma_x\sigma_y\rangle_{G,J}.
\end{equation}
\end{Prop}
\begin{proof} For the first inequality, it suffices to observe that
\begin{align}
\begin{aligned}
	\langle \sigma_x\sigma_y\rangle_{G,J,h}=\sum_{\partial\gamma=\{x,y\}}\rho_{G,J,h}(\gamma)&=K_{G,J,h}(x,y)+\sum_{\substack{\gamma_1:x\rightarrow \mathfrak{g}\\\gamma_2:y\rightarrow \mathfrak{g}}}\rho_{G,J,h}(\gamma_1)\rho_{G\setminus\overline{\gamma_1},J,h}(\gamma_2)
	\\&=K_{G,J,h}(x,y)+\sum_{\gamma_1:x\rightarrow\mathfrak{g}}\rho_{G,J,h}(\gamma_1)\langle \sigma_y\rangle_{G\setminus\overline{\gamma_1},J,h}
	\\&\leq K_{G,J,h}(x,y)+\langle \sigma_x\rangle_{G,J,h}\langle \sigma_y\rangle_{G,J,h},
	\end{aligned}
\end{align}
where we used Proposition \ref{prop:backbone} in the first line and Griffiths' inequality in the last line. 

 The second inequality follows after observing that, since every $\gamma$ contributing to $K_{G,J,h}$ satisfies $\gamma\subset V$, \eqref{eq:propbackb1} implies that
	$\rho_{G,J,h}(\gamma)\leq \rho_{G,J}(\gamma)$, so that
	\begin{equation}
		K_{G,J,h}(x,y)=\sum_{\gamma:x\rightarrow y}\rho_{G,J,h}(\gamma)\leq \sum_{\gamma: x\rightarrow y}\rho_{G,J}(\gamma)=\langle \sigma_x\sigma_y\rangle_{G,J}.
	\end{equation}
This concludes the proof.
\end{proof}

It is unclear which bound in Proposition \ref{prop:trunc vs K vs untrunc} is sharper. The following result suggests that, on average, the two-point function $K$ behaves like the truncated two-point function. Because of our convention on $\prec$, the lower bound in \eqref{eq: prop sum K 1} differs slightly (by a factor at most $2^{2d}$) from that of \cite{AizenmanFernandezCriticalBehaviorMagnetization1986}.
\begin{Prop}[{\hspace{1pt}\cite[Proposition~4.8]{AizenmanFernandezCriticalBehaviorMagnetization1986}}]\label{prop: sum of K} Let $d\geq 2$. Let $G=(V,E)$ be a subgraph of $\mathbb Z^d$. For every $J\in (\mathbb R^+)^E$, every $h\geq 0$, and every $x\in V$,
\begin{equation}\label{eq: prop sum K 1}
	\sum_{y\in V}K_{G,J,h}(x,y)\frac{2^{-2d}\tanh(h)}{1+\langle \sigma_y\rangle_{G,J,h}\tanh(h)}\leq \langle \sigma_x\rangle_{G,J,h}\leq \sum_{y\in V}K_{G,J,h}(x,y)\tanh(h).
\end{equation}
In particular, there exists $h_0>0$ such that, if $h\in (0,h_0]$ and $J_e\leq \beta_c$ for every $e\in E$,  
\begin{equation}\label{eq: prop sum K 2}
	\sum_{y\in V}K_{G,J,h}(x,y)\leq 2^{2d+1}h^{-1}M(\beta_c,h). 
\end{equation}
\end{Prop}
\begin{proof} We prove the result for $G$ finite. It extends to infinite subgraphs of $\mathbb Z^d$ by approximation. We begin by proving \eqref{eq: prop sum K 1}, starting with the upper bound. Using the backbone expansion,
\begin{align}\label{eq: pf sum K 1}
\begin{aligned}
	\langle \sigma_x\rangle_{G,J,h}=\sum_{\gamma:x\rightarrow \mathfrak g}\rho_{G,J,h}(\gamma)&=\sum_{y\in V}\sum_{\gamma':x\rightarrow y}\rho_{G,J,h}(\gamma'\circ (y\fg))
	\\&=\sum_{y\in V}\sum_{\gamma':x\rightarrow y}\rho_{G,J,h}(\gamma')\rho_{G\setminus{\overline{\gamma'}},J,h}((y\fg))
	\\&\leq \sum_{y\in V}K_{G,J,h}(x,y)\tanh(h),
	\end{aligned}
\end{align}
	where $(y\fg)$ denotes the single step $y\fg$, and where we have used Proposition \ref{prop:backbone} in the second and third lines. We now turn to the proof of the lower bound. We extend $J$ to $(\mathbb R^+)^{E_\fg}$ by setting $J_e=h$ for every $e\in E_\fg\setminus E$. We start from the second equality in \eqref{eq: pf sum K 1}. Observe that, by definition of $\prec$, the step $(y\fg)$ erases all $zy$ for $z\sim y$ and $z\in V$. This yields
	\begin{equation}
		\overline{\gamma'\circ (y\fg)}=\overline{\gamma'}\cup \{zy\in E: z\sim y\} \cup \{y\fg\}.
	\end{equation}
	We set $C(y):=\prod_{e\in E: \:e\ni y} \cosh(J_e)$.
	Using the exact formula of $\rho_{G,J,h}(\gamma'\circ (y\fg))$ given by \eqref{eq:formula backbone weight} and the above observation, 
	\begin{align}\label{eq: pf sum K 2}
		\begin{aligned}
			\langle \sigma_x\rangle_{G,J,h}&=\sum_{y\in V}\sum_{\gamma':x\rightarrow y}\Big(\prod_{e\in \gamma'}\tanh(J_e) \prod_{e\in \overline{\gamma'}}\cosh(J_e)\Big) \tanh(h) \:C(y)\cosh(h) \frac{Z^\emptyset_{G\setminus \overline{\gamma'\circ (y\fg)},J,h}}{Z^\emptyset_{G,J,h}}
			\\&=\sum_{y\in V}C(y)\tanh(h)\sum_{\gamma':x\rightarrow y}\rho_{G,J,h}(\gamma')\frac{\cosh(h)Z^\emptyset_{G\setminus \overline{\gamma'\circ (y\fg)},J,h}}{Z^\emptyset_{G\setminus \overline{\gamma'},J,h}}
			\\&=\sum_{y\in V}C(y)\tanh(h)\sum_{\gamma':x\rightarrow y}\rho_{G,J,h}(\gamma')\frac{\cosh(h)}{\langle \exp(h\sigma_y+\sum_{z\sim y}J_{zy}\sigma_z\sigma_y)\rangle_{G\setminus\overline{\gamma'\circ (y\fg)},J,h}},
		\end{aligned}
	\end{align}
	where we used \eqref{eq: relation partition functions} to relate the partition functions of currents to those of the Ising model.
	Now, observe that, setting $C'(y):=\prod_{e\in E: \: e\ni y}\exp(J_e)$,
	\begin{align}\label{eq: pf sum K 3}
	\begin{aligned}
		\langle \exp(h\sigma_y+\sum_{z\sim y}J_{zy}\sigma_z\sigma_y)\rangle_{G\setminus\overline{\gamma'\circ (y\fg)},J,h}&\leq C'(y)\Big(\cosh(h)+\langle \sigma_y\rangle_{G\setminus\overline{\gamma'\circ (y\fg)},J,h}\sinh(h)\Big)\\ &\leq C'(y)\Big(\cosh(h)+\langle \sigma_y\rangle_{G,J,h}\sinh(h)\Big),
		\end{aligned}
	\end{align}
	where the second line follows from Griffiths' inequality.
	Plugging \eqref{eq: pf sum K 3} into \eqref{eq: pf sum K 2}, and observing that $(C(y)/C'(y))\geq 2^{-2d}$ concludes the proof of the lower bound in \eqref{eq: prop sum K 1}.
	
		Finally, \eqref{eq: prop sum K 2} follows from Griffiths' inequality and the lower bound of \eqref{eq: prop sum K 1} once we observe that for $h$ small enough $\frac{\tanh(h)}{1+\tanh(h)}\geq \tfrac{1}{2}h$. This concludes the proof.
\end{proof}

\begin{Rem}\label{rem: K and truncated} In the case $G=\mathbb Z^d$ with $d>4$ and $J\equiv\beta_c$, the combination of Proposition \ref{prop: sum of K}, Theorems \ref{thm: magnetisation} and \ref{thm: susceptibility}, and the lower bound on $\chi(\beta_c,h)$ claimed in Remark \ref{rem:suscept lb} yields
\begin{equation}
	\sum_{x\in \mathbb Z^d}K_{\beta_c,h}(0,x)\asymp \chi(\beta_c,h).
\end{equation}
This suggests that $K_{\beta_c,h}$ is a good approximation of the truncated two-point function at $(\beta_c,h)$, at least in an averaged sense.
\end{Rem}

We now turn to a (simple) new observation: we use the explicit formula for backbone weights given in \eqref{eq:formula backbone weight} to bound the derivative in $J_e$ of $K$, for some $e\in E$. 
\begin{Prop}\label{prop: derivative K} Let $d\geq 2$. There exists $C>0$ such that the following holds. Let $G=(V,E)$ be a subgraph of $\mathbb Z^d$.  Let $J\in (\mathbb R^+)^E$, $h\geq 0$, $S\subset V$, and $x,y\in V$. Let $e_0\in E$ and assume that $J_{e_0}\geq \beta_c/2$. Then,
\begin{equation}
	\frac{\partial}{\partial J_{e_0}}K_{G,J,h}^S(x,y)\leq C\sum_{\gamma:x\rightarrow y}\mathds{1}\{e_0\in \overline{\gamma}\}\rho_{G,J,h}^S(\gamma).
\end{equation}
\end{Prop}
\begin{proof} Again, it suffices to prove the result for $G$ finite. By definition of $K_{G,J,h}^S$, it is enough to bound the $J_{e_0}$-derivative of $\rho_{G,J,h}^S$. Fix $\gamma:x\rightarrow y$ which is consistent, avoids $\mathfrak{g}$, and only visits vertices of $S$. By \eqref{eq:formula backbone weight}, one has
\begin{equation}\label{eq: pf derivative 1}
	\rho_{G,J,h}^S(\gamma)=\prod_{e\in \gamma}\tanh(J_e) \cdot \prod_{e\in \overline{\gamma}}\cosh(J_e)\cdot \frac{Z_{G\setminus \overline{\gamma},J,h}^\emptyset}{Z_{G,J,h}^\emptyset},
\end{equation}
where we use the convention of Proposition \ref{prop:backbone} and set $J_e=h$ for every $e$ of the form $x\fg$, with $x\in V$.
We isolate the factors that depend on $J_{e_0}$ and rewrite \eqref{eq: pf derivative 1} as
\begin{equation}
	\rho_{G,J,h}^S(\gamma)=F_\gamma(J_{e_0})\cdot A \cdot \frac{Z_{G\setminus \overline{\gamma},J,h}^\emptyset}{Z_{G,J,h}^\emptyset},
\end{equation}
where
\begin{equation}
	F_\gamma(J_{e_0})=\big(\mathds{1}_{e_0\notin \gamma}+\tanh(J_{e_0})\mathds{1}_{e_0\in \gamma}\big)\big(\mathds{1}_{e_0\notin \overline{\gamma}}+\cosh(J_{e_0})\mathds{1}_{e_0\in \overline{\gamma}}\big),
\end{equation}
and
\begin{equation}
	A:=\prod_{e\in \gamma\setminus\{e_0\}}\tanh(J_e)\prod_{e\in \overline{\gamma}\setminus \{e_0\}}\cosh(J_e).
\end{equation}
Hence,
\begin{align}\label{eq: pf derivative 2}
\begin{aligned}
	\frac{\partial}{\partial J_{e_0}}\rho_{G,J,h}^S(\gamma)=
	F_\gamma'(&J_{e_0})\cdot A \cdot \frac{Z_{G\setminus \overline{\gamma},J,h}^\emptyset}{Z_{G,J,h}^\emptyset}
	\\&+F_\gamma(J_{e_0})\cdot A \cdot \frac{Z_{G\setminus \overline{\gamma},J,h}^\emptyset}{Z_{G,J,h}^\emptyset}\cdot \Big(\mathds{1}_{e_0\notin\overline{\gamma}}\langle \sigma_{e_0}\rangle_{G\setminus\overline{\gamma},J,h}-\langle \sigma_{e_0}\rangle_{G,J,h}\Big),
	\end{aligned}
\end{align}
where we used \eqref{eq: relation partition functions} to compute the derivative of $\tfrac{Z_{G\setminus \overline{\gamma},J,h}^\emptyset}{Z_{G,J,h}^\emptyset}$, and where $\sigma_{e_0}=\sigma_u\sigma_v$ if $e_0=uv$. By Griffiths' inequality, the second term in the above display is negative. We therefore focus on the first one. Observe that
\begin{align}\label{eq: pf derivative 3}
\begin{aligned}
F_\gamma'(J_{e_0})&=\frac{(\tanh(J_{e_0}))'}{\tanh(J_{e_0})}\mathds{1}_{e_0\in \gamma}F_\gamma(J_{e_0})+\frac{(\cosh(J_{e_0}))'}{\cosh(J_{e_0})}\mathds{1}_{e_0\in \overline{\gamma}}F_\gamma(J_{e_0})
\\&\leq C \mathds{1}_{e_0\in \overline{\gamma}}F_\gamma(J_{e_0}),
\end{aligned}
\end{align}
where $C=C(d)>0$ is given by the assumption that $J_{e_0}\geq \beta_c/2$. Plugging \eqref{eq: pf derivative 3} into \eqref{eq: pf derivative 2} gives
\begin{equation}
	\frac{\partial}{\partial J_{e_0}}\rho_{G,J,h}^S(\gamma)\leq C\mathds{1}_{e_0\in \overline{\gamma}}\rho^S_{G,J,h}(\gamma).
\end{equation}
Summing the above displayed equation over $\gamma:x\rightarrow y$ avoiding $\mathfrak{g}$ and visiting only vertices of $S$ concludes the proof.
\end{proof}

\section{A Simon--Lieb-type inequality for $K$}\label{sec:SL for K}
By Proposition \ref{prop:trunc vs K vs untrunc}, $K_{\beta,h}(0,x)$ provides an upper bound on $\langle\sigma_0;\sigma_x\rangle_{\beta,h}$. The main advantage of the former compared to the latter is that it satisfies a form of Simon--Lieb inequality. Below, we let $\partial^{\textup{ext}}S:=\{y\in \mathbb Z^d\setminus S: \exists x\in S, \: x\sim y\}$ and $\overline{S}=S\cup \partial^{\textup{ext}}S$.

\begin{Lem}[A Simon--Lieb-type inequality for $K$]\label{lem: SL for K} Let $S\subset \mathbb Z^d$ be finite with $0\in S$, and let $E\subset E(\mathbb Z^d)$. For every $x\notin S$,
\begin{equation}
	K_{E^c,\beta,h}(0,x)\leq \Psi_{E^c,\beta,h}(S)\cdot \max_{v\in \partial^{\textup{ext}}S}\max_{\substack{\gamma:0\rightarrow v\\\gamma\subset\overline{S}}}K_{(E\cup\overline{\gamma})^c,\beta,h}(v,x),
\end{equation}
with \begin{equation}
	\Psi_{E^c,\beta,h}(S):=\sum_{\substack{u\in S\\v\notin S\\u\sim v}}K_{E^c,\beta,h}^S(0,u)\cdot \beta,
\end{equation}
where we abused notation by viewing $E^c$ as the graph $(\mathbb Z^d_\fg,E(\mathbb Z^d_\fg)\setminus E)$. 
\end{Lem}
\begin{proof} Let $\gamma:0\rightarrow x$ be a consistent path contributing to $K_{E^c,\beta,h}(0,x)$. Since $0\in S$ and $x\notin S$, we may write (in a unique way) $\gamma=\gamma_1\circ(uv)\circ\gamma_2$ where $u\in S$, $v\notin S$ with $u\sim v$, $\gamma_1:0\rightarrow u\subset S$ and $\gamma_2:v\rightarrow x$. Using this decomposition and Proposition \ref{prop:backbone} yields
\begin{align}
\begin{aligned}
	K_{E^c,\beta,h}(0,x)&=\sum_{\substack{u\in S\\v\notin S\\ u\sim v}}\sum_{\gamma_1:0\rightarrow u}\rho_{E^c,\beta,h}^S(\gamma_1)\rho_{(E\cup \overline{\gamma_1})^c,\beta,h}((uv))K_{(E\cup \overline{\gamma_1\circ (uv)})^c,\beta,h}(v,x)
	\\&\leq \sum_{\substack{u\in S\\v\notin S\\u\sim v}}K_{E^c,\beta,h}^S(0,u)\, \beta  \max_{v\in \partial^{\textup{ext}}S}\max_{\substack{\gamma:0\rightarrow v\\\gamma\subset \overline{S}}}K_{(E\cup\overline{\gamma})^c,\beta,h}(v,x),
	\end{aligned}
\end{align}
		which concludes the proof.
\end{proof}

\begin{Rem}\label{rem:comp psi phi}
(1) Observe that $\Psi_{E^c,\beta,h}(S)$ can be seen as the counterpart of $\varphi_\beta(S)$ in the setting where $h=0$. Because of Lemma \ref{lem: SL for K}, this quantity is (perhaps surprisingly) more natural to consider than 
\begin{equation}
	\sum_{\substack{u\in S\\v\notin S\\u \sim v}}\langle \sigma_0;\sigma_u\rangle_{S,\beta,h}\, \beta.
\end{equation}
Moreover, combining \eqref{eq: K^S vs K_S}, Proposition \ref{prop:trunc vs K vs untrunc}, and Griffiths' inequality gives that, for every $S\subset \mathbb Z^d$ and $E\subset E(\mathbb Z^d)$ as above, one has
\begin{equation}\label{eq:BOUND PSI BY PHI}
	\Psi_{E^c,\beta,h}(S)\leq \varphi_\beta(S).
\end{equation}
%
%

(2) Due to a lack of monotonicity in the volume for the two-point function $K$, we have to keep the restriction to the complement of $\overline{\gamma}$ in Lemma \ref{lem: SL for K}. This explains why we parametrised $\Psi$ in terms of the complement of an edge set $E$: in iterations of Lemma \ref{lem: SL for K}, we will see quantities of the form $\Psi_{E^c,\beta,h}(S)$, with an \emph{evolving} set $E\subset E(\mathbb Z^d)$. This is an additional technical difficulty compared to the case $h=0$.
\end{Rem}

Recall from the introduction that the Simon--Lieb inequality provides a way to upper bound $\xi(\beta,0)$ when $\beta<\beta_c$, by finding a finite set $S$ containing $0$ such that $\varphi_\beta(S)<1$. Lemma \ref{lem: SL for K}, when combined with Proposition \ref{prop:trunc vs K vs untrunc}, allows to show that a similar mechanism holds when $h>0$.

\begin{Prop}\label{prop:sufficient condition to exp decay} Let $d\geq 4$. There exist $c,C>0$ such that the following holds. Let $(\beta,h)\in \mathcal Q^*$. Assume that there exists a finite set $S$ containing $0$ and such that
\begin{equation}\label{eq: assumption to get exp decay}
	\sup_{E\subset E(\mathbb Z^d)}\Psi_{E^c,\beta,h}(S)\leq \frac{1}{2}.
\end{equation}
Then, for every $x\in \mathbb Z^d$,
\begin{equation}
	\sup_{E\subset E(\mathbb Z^d)}K_{E^c,\beta,h}(0,x)\leq \frac{C}{1\vee |x|^{d-2}}\exp\left(-c\frac{|x|}{\operatorname{diam}(S)+1}\right).
\end{equation}
\end{Prop}
\begin{proof} Fix $E_0\subset E(\mathbb Z^d)$. Assume first that $|x|\leq 10(\operatorname{diam}(S)+1)$. By Proposition \ref{prop:trunc vs K vs untrunc}, Griffiths' inequality, and \eqref{eq:IRB}, one has \begin{equation}\label{eq:pf suff 1}
	K_{E_0^c,\beta,h}(0,x)\leq  \langle \sigma_0\sigma_x\rangle_{\beta_c}\leq \frac{C_1}{1\vee |x|^{d-2}}.
\end{equation}
We now assume that $|x|>10(\operatorname{diam}(S)+1)$. Iterating Lemma \ref{lem: SL for K} $k$ times with translates of $S$, where $k:=\lfloor\tfrac{|x|}{2(\operatorname{diam}(S)+1)}\rfloor$, yields
\begin{align}	
\begin{aligned}
	K_{E_0^c,\beta,h}(0,x)&\leq \overline{\Psi}_{\beta,h}(S)^k\cdot \max_{v\in \Lambda_{k(\operatorname{diam}(S)+1)}}\max_{\substack{\gamma:0\rightarrow v\subset \Lambda_{k(\operatorname{diam}(S)+1)} \\\textup{consistent}}}K_{(E_0\cup\overline{\gamma})^c,\beta,h}(v,x)\\&\lesssim 2^{-k}\cdot  \frac{1}{1\vee |x|^{d-2}},
\end{aligned}
\end{align}
where $\overline{\Psi}_{\beta,h}(S):=\sup_{E\subset E(\mathbb Z^d)}\Psi_{E^c,\beta,h}(S)$, and
where we used again Proposition \ref{prop:trunc vs K vs untrunc}, Griffiths' inequality, and \eqref{eq:IRB}, to obtain that, for $v,\gamma$ as above, 
\begin{equation}
K_{(E_0\cup\overline{\gamma})^c,\beta,h}(v,x)\leq \langle \sigma_v\sigma_x\rangle_{\beta_c}\lesssim |v-x|^{2-d}\lesssim |x|^{2-d}.
\end{equation} 
Note that we have a supremum over $E\subset E(\mathbb Z^d)$ above: this is because, in every iteration of Lemma \ref{lem: SL for K}, the ``evolving set'' $E$ is changed into $E\cup \overline{\gamma}$ for some consistent path $\gamma\subset \mathbb Z^d$, the latter edge set being a subset of $E(\mathbb Z^d)$ by definition of $\prec$. This concludes the proof.
\end{proof}

In the case $h=0$, it is not clear how to control $\varphi_\beta(S)$ for a fixed set $S$. In \cite{DuminilPanis2024newLB}, the authors constructed a \emph{random} set $\mathcal S$ for which, on average, $\varphi_\beta(\mathcal S)$ is ``small'', with explicit control of $\operatorname{diam}(\mathcal S)$ in terms of $\beta$. This motivates the following extension of Proposition \ref{prop:sufficient condition to exp decay}.

\begin{Prop}\label{prop:sufficient condition to exp decay random} Let $d\geq 4$. There exist $c,C>0$ such that the following holds. Let $(\beta,h)\in \mathcal Q^*$ and $n \geq 1$. Let $\mathcal S_n$ be a random subset of $\Lambda_n$ that almost surely contains $0$. Let $\mathbb{P}$ denote the law of $\mathcal{S}_n$, and $\mathbb{E}$ the corresponding expectation. Assume that 
\begin{equation}\label{eq: random assumption to get exp decay}
	\sup_{E\subset E(\mathbb Z^d)}\mathbb E[\Psi_{E^c,\beta,h}(\mathcal S_n)]\leq \frac{1}{2}.
\end{equation}
Then, for every $x\in \mathbb Z^d$,
\begin{equation}
	\sup_{E\subset E(\mathbb Z^d)}K_{E^c,\beta,h}(0,x)\leq \frac{C}{1\vee |x|^{d-2}}\exp\left(-c\frac{|x|}{n}\right).
\end{equation}
\end{Prop}

\begin{proof} Fix $E_0\subset E(\mathbb Z^d)$. If $|x|\leq 10n$, we may apply the same reasoning as in \eqref{eq:pf suff 1}. We therefore assume that $|x|>10n$.  Applying Lemma \ref{lem: SL for K} to a sample of $\mathcal S_n$ yields
\begin{align}
\begin{aligned}
	K_{E_0^c,\beta,h}(0,x)&\leq \Psi_{E_0^c,\beta,h}(\mathcal S_n)\cdot \max_{v\in \partial^{\textup{ext}}\mathcal S_n}\max_{\substack{\gamma:0\rightarrow v\\\gamma\subset\overline{\mathcal S_n}}}K_{(E_0\cup\overline{\gamma})^c,\beta,h}(v,x)\\&\leq  \Psi_{E_0^c,\beta,h}(\mathcal S_n)\cdot \max_{v\in \Lambda_{n+1}}\max_{\substack{\gamma:0\rightarrow v\\\gamma\subset\Lambda_{n+1}}}K_{(E_0\cup\overline{\gamma})^c,\beta,h}(v,x)
	\end{aligned}
\end{align}
Averaging over $\mathcal S_n$ and taking the supremum then gives
\begin{equation}
	K_{E_0^c,\beta,h}(0,x)\leq \sup_{E\subset E(\mathbb Z^d)}\mathbb E[\Psi_{E^c,\beta,h}(\mathcal S_n)]\cdot \max_{v\in \Lambda_{n+1}}\max_{\substack{\gamma:0\rightarrow v\\\gamma\subset\Lambda_{n+1}}}K_{(E_0\cup\overline{\gamma})^c,\beta,h}(v,x),
\end{equation}
from which one can reproduce the argument used in the proof of Proposition \ref{prop:sufficient condition to exp decay}.
\end{proof}

As a consequence of Propositions \ref{prop:trunc vs K vs untrunc} and \ref{prop:sufficient condition to exp decay random}, the upper bound of Theorem \ref{thm:main} follows if we find a random subset $\mathcal S_n$ of $\Lambda_n$ with $n\lesssim \ell(\beta,h)$ that satisfies \eqref{eq: random assumption to get exp decay}.

\begin{Thm}\label{thm:finding the good random S} Let $d>4$. There exist $C,\delta >0$ such that the following holds for every $(\beta,h)\in \mathcal Q_\delta$. There exist $n\geq 1$ and a random set $\mathcal S_n\subset \Lambda_n$ of law $\mathbb P$ that almost surely contains $0$ such that the following holds:
\begin{enumerate}
	\item[$(i)$] $n\leq C\ell(\beta,h)$;
	\item[$(ii)$] $\sup_{E\subset E(\mathbb Z^d)}\mathbb E[\Psi_{E^c,\beta,h}(\mathcal S_n)]\leq \tfrac{1}{2}$, where $\mathbb E$ denotes the expectation with respect to $\mathbb P$.
\end{enumerate}
\end{Thm}
 
The upper bound in Theorem \ref{thm:main} follows readily from Theorem \ref{thm:finding the good random S}.
 \begin{proof}[Proof of the upper bound in Theorem~\textup{\ref{thm:main}}] Let $C_1,\delta$ be given by Theorem \ref{thm:finding the good random S}. Let $(\beta,h)\in \mathcal Q_\delta$. By Theorem \ref{thm:finding the good random S}, there exist $n\leq C_1\ell(\beta,h)$ and a random set $\mathcal S_n\subset \Lambda_n$ of law $\mathbb P$ which almost surely contains $0$, such that
 \begin{equation}
 	\sup_{E\subset E(\mathbb Z^d)}\mathbb E[\Psi_{E^c,\beta,h}(\mathcal S_n)]\leq \frac{1}{2}.
 \end{equation}
Therefore, by Propositions \ref{prop:trunc vs K vs untrunc} and \ref{prop:sufficient condition to exp decay random}, there exist $c_2,C_2>0$ such that, for every $x\in \mathbb Z^d$,
 \begin{align}\label{eq:proof upper bound 1}
 \begin{aligned}
	\langle \sigma_0;\sigma_x\rangle_{\beta,h}&\leq K_{\beta,h}(0,x)\leq \frac{C_2}{1\vee |x|^{d-2}}\exp\left(-c_2\frac{|x|}{n}\right)\leq \frac{C_2}{1\vee |x|^{d-2}}\exp\left(-\frac{c_2}{C_1}\frac{|x|}{\ell(\beta,h)}\right).
	\end{aligned}
\end{align}
This concludes the proof.
\end{proof}

 A similar observation can be made in dimension $d=4$ to derive the upper bound in Theorem \ref{thm:main bis}.  
 \begin{Thm}\label{thm:finding the good random S d=4} Let $d=4$ and $\varepsilon>0$. There exist $C,\delta >0$ such that the following holds for every $(\beta,h)\in \mathcal Q_\delta$. There exist $n\geq 1$ and a random set $\mathcal S_n\subset \Lambda_n$ of law $\mathbb P$ that almost surely contains $0$ such that the following holds:
\begin{enumerate}
	\item[$(i)$] $n\leq C\ell(\beta,h)(\log \ell(\beta,h))^{2+\varepsilon}$;
	\item[$(ii)$] $\sup_{E\subset E(\mathbb Z^d)}\mathbb E[\Psi_{E^c,\beta,h}(\mathcal S_n)]\leq \tfrac{1}{2}$, where $\mathbb E$ denotes the expectation with respect to $\mathbb P$.
\end{enumerate}
\end{Thm}

\begin{proof}[Proof of the upper bound in Theorem~\textup{\ref{thm:main bis}}] The proof is identical to that of the upper bound in Theorem \ref{thm:main}, using Theorem \ref{thm:finding the good random S d=4} in place of Theorem \ref{thm:finding the good random S}. This replaces $\ell(\beta,h)$ by $\ell(\beta,h)(\log \ell(\beta,h))^{2+\varepsilon}$ in \eqref{eq:proof upper bound 1}.
\end{proof}

We now turn to the proofs of 
Theorems~\ref{thm:finding the good random S} 
and~\ref{thm:finding the good random S d=4}. As explained in the 
introduction, the approach builds on the information available at $h=0$.

\section{Extrapolation principle: comparison to the $h=0$ case}\label{sec:extrapolation}

In this section, we prove Theorems~\ref{thm:finding the good random S} 
and~\ref{thm:finding the good random S d=4}. The thermal regime $\ell(\beta,h)=(\beta_c-\beta)^{-1/2}$ follows from 
\cite{DuminilPanis2024newLB}, as detailed below. To treat the magnetic regime $\ell(\beta,h)=h^{-1/3}$ we implement the strategy described in the introduction. We first use Proposition \ref{prop: derivative K} to give an explicit bound on the $\beta$-derivative of $\mathbb E[\Psi_{E^c,\beta,h}(\mathcal S_n)]$, see Lemma \ref{lem: derivative}. We then define the \emph{good} random set $\mathcal S_n$---first introduced in \cite{DuminilPanis2024newLB}---and present the relevant properties it satisfies, see Proposition \ref{prop:properties of the random set}. The proof of this result is postponed to the next section. Finally, we complete the proofs of Theorems \ref{thm:finding the good random S} and \ref{thm:finding the good random S d=4}.

\subsection{Derivative computation}

Let $(\beta,h)\in \mathcal Q$, $\beta'\in [0,\beta]$, and fix $E\subset E(\mathbb Z^d)$. Let $n\geq 1$ be even, and fix a random set $\mathcal S_n\subset \Lambda_n$ of law $\mathbb P$, which almost surely contains $0$. In the next result, we rely on Proposition \ref{prop: derivative K} to compare $\mathbb E[\Psi_{E^c,\beta,h}(\mathcal S_n)]$ to the analogous quantity at parameters $(\beta',0)$. For technical reasons, we modify $\beta$ to $\beta'$ only in the box $\Lambda_{n/2}$.

Before stating the result, we need some notation. For $t\in [0,1]$, we define $J(t)\in (\mathbb R^+)^{E(\mathbb Z^d)}$ as follows: 
\begin{equation}\label{eq: def J}
	J_e(t):=\left\{
			\begin{array}{ll}
				\beta'+t(\beta-\beta') & \mbox{if } e\in E(\Lambda_{n/2}), \\
				\beta & \mbox{otherwise}.
			\end{array}
			\right.
\end{equation}
If $S\subset \mathbb Z^d$ and $u\in \mathbb Z^d$, we let
\begin{equation}
	\varphi_{\beta_c}^u(S)=\sum_{\substack{x\in S\\y\notin S\\x \sim y}}\langle \sigma_u\sigma_x\rangle_{S,\beta_c}\, \beta_c.
\end{equation}
Observe that $\varphi^u_{\beta_c}(S)=0$ is $u \notin S$.

\begin{Lem}\label{lem: derivative} Let $d\geq 4$. There exists $C>0$ such that the following holds. For every $\tfrac{\beta_c}{2}\leq \beta'\leq \beta\leq \beta_c$ and every $h\geq 0$,
\begin{multline}\label{eq: derivative}
	\mathbb E[\Psi_{E^c,\beta,h}(\mathcal S_n)]\leq 2\mathbb E[\varphi_{J(0)}(\mathcal S_n)]\\+C(\beta-\beta') \sup_{0\leq t\leq 1}\sum_{u\in \Lambda_{n/2}}K_{E^c,J(t),h}(0,u)\mathbb E[\varphi_{\beta_c}^u(\mathcal S_n)].
\end{multline}
	
\end{Lem}

\begin{proof} We abuse notations slightly and define
\begin{equation}
\Psi_{E^c,J(t),h}(\mathcal S_n)=\sum_{\substack{x\in \mathcal S_n\\y\notin \mathcal S_n\\x\sim y}}K_{E^c,J(t),h}^{\mathcal S_n}(0,x)\cdot \beta,
\end{equation}
i.e.~$J(t)$ only affects the two-point function $K$, not the interaction multiplying the sum above.
By the fundamental theorem of calculus, 
\begin{equation}\label{eq:pf der psi 1}
	\mathbb E[\Psi_{E^c,\beta,h}(\mathcal S_n)]=\mathbb E[\Psi_{E^c,J(0),h}(\mathcal S_n)]+\int_0^1\mathbb E\Big[\frac{\partial}{\partial t}\Psi_{E^c,J(t),h}(\mathcal S_n)\Big]\mathrm{d}t,
\end{equation}
where we used that $\frac{\partial}{\partial t}\mathbb E[\Psi_{E^c,J(t),h}(\mathcal S_n)]=\mathbb E[\frac{\partial}{\partial t}\Psi_{E^c,J(t),h}(\mathcal S_n)]$.
By \eqref{eq:BOUND PSI BY PHI}, one has
\begin{equation}\label{eq:pf der psi 2}
	\mathbb E[\Psi_{E^c,J(0),h}(\mathcal S_n)]\leq (\beta/\beta')\mathbb E[\varphi_{J(0)}(\mathcal S_n)]\leq 2\mathbb E[\varphi_{J(0)}(\mathcal S_n)].
\end{equation}
Fix $S\subset \Lambda_n$ which contains $0$. By Proposition \ref{prop: derivative K} and our assumption on $\beta'$, there exists $C_1=C_1(d)>0$ such that
\begin{equation}\label{eq:pf der psi 3}
	\frac{\partial}{\partial t}\Psi_{E^c,J(t),h}(S)\leq C_1(\beta-\beta')\sum_{\substack{x\in S\\y\notin  S\\ x\sim y}}\sum_{e\in E(\Lambda_{n/2})}\sum_{\gamma:0\rightarrow x}\mathds{1}\{e\in \overline{\gamma}\}\rho_{E^c,J(t),h}^{S}(\gamma).
\end{equation}
We now make the following observation: if $e=uv\in \overline{\gamma}$, then either $u\in \gamma$ or $v\in \gamma$. If the former case holds, we can write $\gamma=\gamma_1\circ\gamma_2$ where $\gamma_1:0\rightarrow u$ visits $u$ exactly once, and $\gamma_2: u\rightarrow x$. Combining this observation with Proposition \ref{prop:backbone} gives
\begin{align}\label{eq:pf der psi 4}
\begin{aligned}
	\sum_{e\in E(\Lambda_{n/2})}\sum_{\gamma:0\rightarrow x}\mathds{1}\{e\in \overline{\gamma}\}&\rho_{E^c,J(t),h}^{S}(\gamma)
	\\&\leq 2d \sum_{u\in \Lambda_{n/2}}\sum_{\gamma_1:0\rightarrow u}\rho_{E^c,J(t),h}^S(\gamma_1)K_{(E\cup\overline{\gamma}_1)^c,J(t),h}^S(u,x)
\\&\leq 2d\sum_{u\in \Lambda_{n/2}}K_{E^c,J(t),h}(0,u)\langle \sigma_u\sigma_x\rangle_{S,\beta_c},
\end{aligned}
\end{align}
where we used \eqref{eq:propbackb1} and Griffiths' inequality to argue that \begin{equation}
K_{(E\cup\overline{\gamma}_1)^c,J(t),h}^S(u,x)\leq\langle \sigma_{u}\sigma_x\rangle_{(S,E(S)\setminus(E\cup\overline{\gamma}_1)),J(t)} \leq \langle \sigma_u\sigma_x\rangle_{S,\beta_c}.
\end{equation}
The combination of \eqref{eq:pf der psi 3} and \eqref{eq:pf der psi 4} yields
\begin{equation}
	\mathbb E\Big[\frac{\partial}{\partial t}\Psi_{E^c,J(t),h}(\mathcal S_n)\Big]\leq 2dC_1(\beta-\beta') \sup_{s\in [0,1]} \sum_{u\in \Lambda_{n/2}}K_{E^c,J(s),h}(0,u)\mathbb E[\varphi_{\beta_c}^u(\mathcal S_n)],
\end{equation}
which, when plugged into \eqref{eq:pf der psi 1}, gives the desired result.
\end{proof}

In view of Lemma~\ref{lem: derivative}, the random set $\mathcal{S}_n$ 
must be chosen so that the right-hand side of~\eqref{eq: derivative} can 
be made smaller than $\tfrac{1}{2}$ for an appropriate choice of $\beta'$ and $n\lesssim\ell(\beta,h)$.

\subsection{Definition of the random set $\mathcal S_n$}\label{sec:definition of S_n}

We now describe how to construct the random set $\mathcal S_n$ mentioned at the beginning of the section. We first import notation from \cite{AizenmanDuminilTassionWarzelEmergentPlanarity2019,DuminilPanis2024newLB}. For each $n\geq 1$, we let $\mathbb H_n:=\{x\in \mathbb Z^d: x_1=n\}$, and denote by $\mathcal R_n$ the reflection with respect to $\mathbb H_n$. 

We consider \emph{folded} single random currents on $\mathbb Z^d$ (see \cite{aizenman2025geometric} for a review of this ``trick''), and partition the edge-set $E(\mathbb Z^d)$ into three disjoint sets:
\begin{align*}
    E^{-}(\mathbb H_n)&:=\Big\{ uv\in E(\mathbb Z^d) :\textup{at least one endpoint is strictly on the left of $\mathbb H_n$}\Big\},\\
    E^{+}(\mathbb H_n)&:=\Big\{ uv\in E(\mathbb Z^d) :\textup{at least one endpoint is strictly on the right of $\mathbb H_n$}\Big\},\\
    E^{0}(\mathbb H_n)&:=E(\mathbb Z^d)\setminus\left(E^{-}(\mathbb H_n)\cup E^{+}(\mathbb H_n)\right).
\end{align*}
Decompose a current $\n$ on $\mathbb Z^d$ into its restrictions $\n^{-}$, $\n^{+}$ and $\n^0$ to the above three subsets of $E(\mathbb Z^d)$. It is convenient to consider the \emph{multigraph} representation of these objects, which we denote by $\mathcal N^-$, $\mathcal N^+$, and $\mathcal N^0$, respectively. Consider the multigraph $\mathcal{M}_n$
 obtained by taking the union of the multigraph $\mathcal{N}^{-}$ and the reflection $\mathcal{R}_n(\mathcal{N}^+)$ of the multigraph $\mathcal{N}^+$. The above definitions are all with respect to the direction $\mathbf e_1=(1,0,\ldots,0)\in \mathbb Z^d$. In order to extend them to other directions, we write $\mathcal{M}_n(+\mathbf{e}_1):=\mathcal{M}_n$, $\mathcal R_n(+\mathbf{e}_1):=\mathcal R_n$, $\mathbb H_n(+\mathbf{e}_1):=\mathbb H_n$, and define
\begin{equation}
   \mathcal{S}_n(+\mathbf{e}_1):=\Big\lbrace x\in \Lambda_n: \: x \overset{\mathcal M_n(+\mathbf e_1)}{\centernot\longleftrightarrow} \mathbb H_n(+\mathbf{e}_1)\Big\rbrace.\end{equation}
\noindent Similarly, define $\mathbb H_n(\pm \mathbf{e}_i)$, $\mathcal R_n(\pm \mathbf{e}_i)$, $\mathcal{M}_n(\pm \mathbf{e}_i)$, and $\mathcal{S}_n(\pm\mathbf{e}_i)$ for $1\leq i \leq d$ in the other $2d-1$ directions. Let
\begin{equation}\label{eq:def S_n}
    \mathcal{S}_n:=\bigcap_{1\leq i \leq d}\left(\mathcal{S}_n(+\mathbf{e}_i)\cap \mathcal{S}_n(-\mathbf{e}_i)\right).
\end{equation}
Observe that $\mathcal S_n\subset \Lambda_{n-1}$.
The random set $\mathcal S_n$ is obtained by sampling $\n$ according to $\mathbf P_{\beta_c}^\emptyset$. The analysis of $\mathcal{S}_n$ relies on the fact that $\mathcal{M}_n$ 
can be viewed as the superposition of two independent currents, enabling 
the use of switching techniques; see Section~\ref{sec: properties of 
random set} for details. In particular, $\mathcal{S}_n$ satisfies properties needed to control the right-hand side of~\eqref{eq: derivative}. Recall the definition of $J(t)$ from \eqref{eq: def J}. 
\begin{Prop}[Properties of the random set $\mathcal S_n$]\label{prop:properties of the random set} Let $d\geq 4$. There exist $c,C>0$ such that the following holds.  For every $(\beta,h)\in \mathcal Q$, every $\beta'\in [0,\beta)$, and every $n$ large enough,
\begin{equation}\label{eq:prop 1}
	\mathbf P_{\beta_c}^\emptyset[0\in \mathcal S_n]\geq \frac{1}{2},
\end{equation}
\begin{equation}\label{eq:prop 2}
	\mathbf E_{\beta_c}^{\emptyset}[\mathds{1}_{0\in \mathcal S_n}\varphi_{J(0)}(\mathcal S_n)]\leq \frac{1}{16}+ C\exp\left(-c\frac{(\beta_c-\beta')^{1/2}}{|\log (\beta_c-\beta')|^{\mathds{1}_{d=4}}} n\right)\cdot (\log n)^{\mathds{1}_{d=4}},
\end{equation}
\begin{equation}\label{eq:prop 3}
	\max_{u\in \Lambda_{n/2}}\mathbf E_{\beta_c}^\emptyset[\mathds{1}_{0,u\in \mathcal S_n}\varphi_{\beta_c}^u(\mathcal S_n)]\leq C(\log n)^{\mathds{1}_{d=4}}.
\end{equation}
\end{Prop}

We postpone the proof of this result to the next section and first conclude the proofs of Theorems \ref{thm:finding the good random S} and \ref{thm:finding the good random S d=4}. Before doing so, we state a consequence of Lemma \ref{lem: derivative} and Proposition \ref{prop:properties of the random set}.

\begin{Coro}\label{coro:bound on psi} Let $d\geq 4$. There exist $c,C,h_0>0$ such that the following holds. For every $\tfrac{\beta_c}{2}\leq \beta'< \beta\leq \beta_c$, every $h\in (0,h_0]$, and every $n$ large enough,
\begin{multline}
	\sup_{E\subset E(\mathbb Z^d)}\mathbf E_n[\Psi_{E^c,\beta,h}(\mathcal S_n)]\leq \frac{1}{4}+ C\exp\left(-c\frac{(\beta_c-\beta')^{1/2}}{|\log (\beta_c-\beta')|^{\mathds{1}_{d=4}}}n\right)\cdot (\log n)^{\mathds{1}_{d=4}}
	\\+C(\beta-\beta')\cdot h^{-2/3}|\log h|^{\mathds{1}_{d=4}}\cdot (\log n)^{\mathds{1}_{d=4}},
\end{multline}
where $\mathbf P_n:=\mathbf P^\emptyset_{\beta_c}[\:\cdot \mid 0\in \mathcal S_n]$.
\end{Coro}
\begin{proof} Fix $\beta',\beta,h$ as in the statement. Let $c_1,C_1>0$ be given by Proposition \ref{prop:properties of the random set}. Plugging \eqref{eq:prop 1} into \eqref{eq:prop 2}--\eqref{eq:prop 3} gives, for $n$ large enough,
\begin{equation}\label{eq:cor pf 1}
	\mathbf E_n[\varphi_{J(0)}(\mathcal S_n)]\leq \frac{1}{8}+2C_1\exp\left(-c_1\frac{(\beta_c-\beta')^{1/2}}{|\log (\beta_c-\beta')|^{\mathds{1}_{d=4}}} n\right)\cdot (\log n)^{\mathds{1}_{d=4}},
\end{equation}
and
\begin{equation}\label{eq:cor pf 2}
		\max_{u\in \Lambda_{n/2}}\mathbf E_n[\varphi_{\beta_c}^u(\mathcal S_n)]\leq 2C_1(\log n)^{\mathds{1}_{d=4}}.
\end{equation}
Applying Lemma \ref{lem: derivative} to $\mathbb P=\mathbf P_n$ gives $C_2>0$ such that
\begin{multline}\label{eq:cor pf 3}
	\sup_{E\subset E(\mathbb Z^d)}\mathbf E_n[\Psi_{E^c,\beta,h}(\mathcal S_n)]\leq \frac{1}{4}+4C_1\exp\left(-c_1\frac{(\beta_c-\beta')^{1/2}}{|\log (\beta_c-\beta')|^{\mathds{1}_{d=4}}} n\right)\cdot (\log n)^{\mathds{1}_{d=4}}
	\\+C_2(\beta-\beta') (\log n)^{\mathds{1}_{d=4}} \sup_{E\subset E(\mathbb Z^d)}\sup_{0\leq t\leq 1}\sum_{u\in \Lambda_{n/2}}K_{E^c,J(t),h}(0,u),
\end{multline}
where we used \eqref{eq:cor pf 1} and \eqref{eq:cor pf 2}. Now, thanks to Proposition \ref{prop: sum of K}, Griffiths' inequality, and Theorem \ref{thm: magnetisation}, if $h$ is small enough,
\begin{equation}\label{eq:cor pf 4}
	\sup_{E\subset E(\mathbb Z^d)}\sup_{0\leq t\leq 1}\sum_{u\in \Lambda_{n/2}}K_{E^c,J(t),h}(0,u)\lesssim h^{-1} M(\beta_c,h)\lesssim h^{-2/3}|\log h|^{\mathds{1}_{d=4}}.
\end{equation} 
		Plugging \eqref{eq:cor pf 4} into \eqref{eq:cor pf 3} concludes the proof.
\end{proof}

\subsection{Proofs of Theorems \ref{thm:finding the good random S}--\ref{thm:finding the good random S d=4}}

We now complete the proofs of Theorems \ref{thm:finding the good random S} and \ref{thm:finding the good random S d=4}, starting with the case $d>4$. Recall that we have set $\mathbf P_n=\mathbf P_{\beta_c}^\emptyset[\: \cdot\mid 0\in \mathcal S_n]$.
\begin{proof}[Proof of Theorem~\textup{\ref{thm:finding the good random S}}] Fix $(\beta,h)\in \mathcal Q^*$ sufficiently close to $(\beta_c,0)$. We use two different arguments depending on whether $\ell(\beta,h)=(\beta_c-\beta)^{-1/2}$ or $\ell(\beta,h)=h^{-1/3}$. 

 Assume first that $\ell(\beta,h)=(\beta_c-\beta)^{-1/2}$. By Theorem \ref{thm:sharplength}, one can find $n\lesssim (\beta_c-\beta)^{-1/2}$ such that there exists $S\subset \Lambda_n$ containing $0$ with $\varphi_{\beta}(S)\leq \tfrac{1}{2}$. Let $\mathcal S_n$ be the random set distributed as $\mathbb P$, the Dirac measure at $S$. By \eqref{eq:BOUND PSI BY PHI},
\begin{equation}
	\sup_{E\subset E(\mathbb Z^d)}\mathbb E[\Psi_{E^c,\beta,h}(\mathcal S_n)]=\sup_{E\subset E(\mathbb Z^d)}\Psi_{E^c,\beta,h}(S)\leq \varphi_{\beta}(S)\leq \frac{1}{2}.
\end{equation}
This concludes the proof in this case.

We now assume that $\ell(\beta,h)=h^{-1/3}$. Let $c_1,C_1>0$ be given by Corollary \ref{coro:bound on psi}. Set $\beta'=\beta-\eta h^{2/3}$, where $\eta>0$ is small enough so that
\begin{equation}\label{eq:pf psi drop d>4 2}
	C_1(\beta-\beta')h^{-2/3}= C_1\eta\leq \frac{1}{8}.
\end{equation}
Observe that 
\begin{equation}\label{eq:pf psi drop d>4 3}
	\beta_c-\beta'\geq \beta-\beta'= \eta h^{2/3}.\end{equation}
For this choice of $\eta$,  we pick $n$ large enough so that
\begin{equation}\label{eq:pf psi drop d>4 4}
	C_1\exp(-c_1(\beta_c-\beta')^{1/2}n)\leq C_1\exp(-c_1\eta^{1/2}h^{1/3}n)\leq \frac{1}{8}.
\end{equation}
One may choose $n\lesssim h^{-1/3}=\ell(\beta,h)$ above. Combining \eqref{eq:pf psi drop d>4 2} and \eqref{eq:pf psi drop d>4 4} with Corollary \ref{coro:bound on psi}, and defining $\mathcal S_n$ to be the random set distributed according to $\mathbb P:=\mathbf P_n$, concludes the proof.
\end{proof}

\begin{proof}[Proof of Theorem~\textup{\ref{thm:finding the good random S d=4}}] Let $\varepsilon>0$.
Fix $(\beta,h)\in \mathcal Q^*$ sufficiently close to $(\beta_c,0)$. Again, we use two different arguments depending on whether $\ell(\beta,h)=(\beta_c-\beta)^{-1/2}$ or $\ell(\beta,h)=h^{-1/3}$. 

 Assume first that $\ell(\beta,h)=(\beta_c-\beta)^{-1/2}$. By Theorem \ref{thm:sharplength}, one can find $n\lesssim (\beta_c-\beta)^{-1/2}|\log (\beta_c-\beta)|$ such that there exists $S\subset \Lambda_n$ containing $0$ with $\varphi_{\beta}(S)\leq \tfrac{1}{2}$. Let $\mathcal S_n$ be the random set distributed as $\mathbb P$, the Dirac measure at $S$. By \eqref{eq:BOUND PSI BY PHI},
\begin{equation}
	\sup_{E\subset E(\mathbb Z^d)}\mathbb E[\Psi_{E^c,\beta,h}(\mathcal S_n)]=\sup_{E\subset E(\mathbb Z^d)}\Psi_{E^c,\beta,h}(S)\leq \varphi_{\beta}(S)\leq \frac{1}{2}.
\end{equation}
This concludes the proof in this case.

We now assume that $\ell(\beta,h)=h^{-1/3}$. Let $c_1,C_1>0$ be given by Corollary \ref{coro:bound on psi}. Set $\beta'=\beta-h^{2/3}|\log h|^{-\alpha}(\log |\log h|)^{-1}$, with $\alpha>0$ to be chosen appropriately. We want to guarantee that, for $h$ sufficiently close to $0$, and $n\asymp\ell(\beta,h)(\log \ell(\beta,h))^{2+\varepsilon}\asymp h^{-1/3}|\log h|^{2+\varepsilon}$,
\begin{equation}\label{eq:pf psi drop d=4 1}
	C_1(\beta-\beta')h^{-2/3}|\log h| \log n=C_1\frac{\log n}{|\log h|^{\alpha-1}\log |\log h|}\leq \frac{1}{8},
\end{equation}
and
\begin{equation}\label{eq:pf psi drop d=4 2}
	C_1\exp\left(-c_1\frac{(\beta_c-\beta')^{1/2}}{|\log (\beta_c-\beta')|}n\right)\log n\leq \frac{1}{8}.
\end{equation}
Observe that, since $\beta_c-\beta'\geq \beta-\beta'$, \eqref{eq:pf psi drop d=4 2} is implied by
\begin{equation}\label{eq:pf psi drop d=4 3}
C_1\exp\left(-c_2\frac{h^{1/3}n}{|\log h|^{1+\alpha/2}(\log |\log h|)^{1/2}} \right)\log n\leq \frac{1}{8},
\end{equation}
for some appropriate $c_2=c_2(c_1)>0$.
Pick $\alpha=2$, and set $n=h^{-1/3}|\log h|^{2+\varepsilon}$. With this choice, \eqref{eq:pf psi drop d=4 3} is implied by
\begin{equation}\label{eq: where eps becomes crucial}
	C_1\exp(-c_3 (\log n)^{\varepsilon/2})\log n\leq \frac{1}{8},
\end{equation}
for some appropriate $c_3=c_3(c_2)>0$,
which holds for $n$ large enough. Moreover, \eqref{eq:pf psi drop d=4 1} is implied by
\begin{equation}
	\frac{C_2}{\log |\log h|}\leq \frac{1}{8},
\end{equation}
for some $C_2=C_2(C_1)>0$, which holds as long as $h$ is close enough to $0$ (or $n$ is large enough). As a result, for this choice of $n$, by Corollary \ref{coro:bound on psi}, and setting $\mathbb P:=\mathbf P_n$, we find that
\begin{equation}
	\sup_{E\subset E(\mathbb Z^d)}\mathbb E[\Psi_{E^c,\beta,h}(\mathcal S_n)]\leq \frac{1}{2},
\end{equation}
	which concludes the proof.
\end{proof}

\section{Properties of $\mathcal S_n$}\label{sec: properties of random set}
In this section, we prove Proposition \ref{prop:properties of the random set}. Recall the notations from Section \ref{sec:definition of S_n}, and set
\begin{equation} 
\mathcal S_n^1:=\mathcal S_n(+\mathbf{e}_1)\cap \Lambda_{n-1}.
\end{equation}
 We begin by recalling some fundamental consequences of the switching lemma derived in \cite{DuminilPanis2024newLB}. 
\begin{Lem}\label{lem:random phi moved} Let $u\in \Lambda_n$. One has
\begin{equation}
	\mathbf E_{\beta_c}^\emptyset[\varphi_{\beta_c}^u(\mathcal S_n)]\leq 2d\sum_{u'\sim_{\theta}u}\sum_{\substack{x,y\in \Lambda_n\\x\sim y}}\mathbf E_{\beta_c}^\emptyset\Big[\mathds{1}\{u',x\in \mathcal S_n^1\}\mathds{1}\{y\connect{\mathcal M_n\:}\mathbb H_n\}\langle \sigma_{u'}\sigma_x\rangle_{\mathcal S_n^1,\beta_c}\Big],
\end{equation}
where $u'\sim_\theta u$ means that $u'$ can be obtained from $u$ by an 
isometry of $\mathbb{Z}^d$ (i.e.\ a composition of coordinate permutations 
and reflections).\end{Lem}
\begin{proof} Let $u\in \Lambda_n$. We follow the computations below \cite[Equation~(2.15)]{DuminilPanis2024newLB}. Observe that, by definition of $\mathcal S_n$ from \eqref{eq:def S_n} and Griffiths' inequality, 
\begin{align}\label{eq:pf lem prop 1}
\begin{aligned}
	\varphi_{\beta_c}^u(\mathcal S_n)=\beta_c\sum_{\substack{x\in \mathcal S_n\\ y\notin \mathcal S_n, \: y \sim x\\ }}\langle \sigma_u\sigma_x\rangle_{\mathcal S_n,\beta_c}&\leq \sum_{\varepsilon\in \{\pm 1\}}\sum_{i=1}^d \sum_{\substack{x\in \mathcal S_n(\varepsilon\mathbf{e}_i)\cap \Lambda_{n-1}\\y\notin \mathcal S_n(\varepsilon\mathbf{e}_i), \: y\sim x}}\langle \sigma_u\sigma_x\rangle_{\mathcal S_n(\varepsilon\mathbf{e}_i)\cap \Lambda_{n-1},\beta_c}\,\beta_c
	\\&=: \sum_{\varepsilon\in \{\pm 1\}}\sum_{i=1}^d F(u;\mathcal S_n(\varepsilon \mathbf{e}_i))\\&\leq \sum_{\varepsilon\in \{\pm 1\}}\sum_{i=1}^d \Big(\sum_{u'\sim_\theta u} F(u';\mathcal S_n(\varepsilon \mathbf{e}_i))\Big)
	\end{aligned}
\end{align}
We now take the expectation above. Since $\mathbf P_{\beta_c}^\emptyset$ is invariant under isometries of $\mathbb Z^d$, the law of $\Big(\sum_{u'\sim_\theta u} F(u';\mathcal S_n(\varepsilon \mathbf{e}_i))\Big)$ is the same as that of $\Big(\sum_{u'\sim_\theta u} F(u';\mathcal S_n(\mathbf{e}_1))\Big)$.
As a result,
\begin{equation}
	\mathbf E_{\beta_c}^\emptyset[\varphi_{\beta_c}^u(\mathcal S_n)]\leq 2d \mathbf E_{\beta_c}^\emptyset\Big[\sum_{u'\sim_\theta u} F(u';\mathcal S_n(\mathbf{e}_1))\Big],
\end{equation}
which concludes the proof.
\end{proof}

\begin{Lem} One has,
\begin{equation}\label{eq:bound of proba 0 notin S_n}
	\mathbf P_{\beta_c}^{\emptyset}[0\notin \mathcal S_n]\leq 2d\langle \sigma_0\sigma_{2n\mathbf{e}_1}\rangle_{\beta_c},
\end{equation}
and, if $a,x,y\in \Lambda_n$, with $x\sim y$, and $\beta\leq \beta_c$,
\begin{multline}\label{eq:switching where we move 0}
	\mathbf E_{\beta_c}^\emptyset\Big[\mathds{1}\{a,x\in \mathcal S_n^1\}\mathds{1}\{y\connect{\mathcal M_n\:}\mathbb H_n\}\langle \sigma_a\sigma_x\rangle_{\mathcal S_n^1,\beta}\Big]\\\leq \Big(\langle \sigma_a\sigma_x\rangle_{\beta}-\langle \sigma_a\sigma_{\mathcal R_n(x)}\rangle_{\beta}\Big)\langle \sigma_y\sigma_{\mathcal R_n(y)}\rangle_{\beta_c}.
	\end{multline}
\end{Lem}
\begin{proof} These results follow from the switching techniques developed in \cite[Lemma~2.3]{DuminilPanis2024newLB}. Equation \eqref{eq:bound of proba 0 notin S_n} is exactly \cite[Equation~(2.12)]{DuminilPanis2024newLB}, while \eqref{eq:switching where we move 0} is a straightforward extension of \cite[Equation~(2.28)]{DuminilPanis2024newLB} (in which $a$ replaces $0$).
\end{proof}

\begin{Coro}\label{coro:bound translate random s_n} Let $d\geq 4$. There exists $C>0$ such that, for every $u\in \Lambda_{n/2}$, 
\begin{equation}
	\mathbf E_{\beta_c}^\emptyset[\varphi_{\beta_c}^u(\mathcal S_n)]\leq C(\log n)^{\mathds{1}_{d=4}}.
\end{equation}
\end{Coro}
\begin{proof} By Lemma \ref{lem:random phi moved} and \eqref{eq:switching where we move 0}, if $u\in \Lambda_{n/2}$,
\begin{equation}\label{eq:proof of prop 2}
	\mathbf E_{\beta_c}^\emptyset[\varphi_{\beta_c}^u(\mathcal S_n)]\leq 2d\sum_{u'\sim_\theta u}\sum_{\substack{x,y\in \Lambda_n\\x\sim y}}\Big(\langle \sigma_{u'}\sigma_x\rangle_{\beta_c}-\langle \sigma_{u'}\sigma_{\mathcal R_n(x)}\rangle_{\beta_c}\Big)\langle \sigma_y\sigma_{\mathcal R_n(y)}\rangle_{\beta_c}.
	\end{equation}
In \cite[Theorem~1.3]{DuminilPanis2024newLB}, the authors rely on several consequences of reflection positivity (and, in particular, \eqref{eq:IRB}) to prove that
\begin{equation}\label{eq:proof of prop 3}
	\sum_{\substack{x,y\in \Lambda_n\\x\sim y}}\Big(\langle \sigma_{0}\sigma_x\rangle_{\beta_c}-\langle \sigma_{0}\sigma_{\mathcal R_n(x)}\rangle_{\beta_c}\Big)\langle \sigma_y\sigma_{\mathcal R_n(y)}\rangle_{\beta_c}\lesssim (\log n)^{\mathds{1}_{d=4}}.\end{equation}
		The same argument as in \cite{DuminilPanis2024newLB} shows that this bound extends when $0$ is replaced by any $u\in \Lambda_{n/2}$. We provide a proof for the sake of completeness. Fix $u'\sim_\theta u$. Divide the sum on the right-hand side of \eqref{eq:proof of prop 2} according to whether $-n\leq x_1\leq \lfloor 3n/4\rfloor $ or  $\lfloor 3n/4\rfloor  < x_1\leq n$.  By \eqref{eq:IRB}, 
\begin{align}\label{eq:proof of prop 3.3}
\begin{aligned}
    \sum_{\substack{x,y\in \Lambda_n\\x_1\leq \lfloor 3n/4\rfloor \\y\sim x}}\Big(\langle \sigma_{u'}\sigma_x\rangle_{\beta_c}-\langle \sigma_{u'}\sigma_{\mathcal{R}_n(x)}\rangle_{\beta_c}\Big)\langle \sigma_y \sigma_{\mathcal{R}_n(y)}\rangle_{\beta_c} &\leq\sum_{\substack{x,y\in \Lambda_n\\ \: x_1\leq \lfloor 3n/4\rfloor \\y\sim x}}\langle\sigma_{u'}\sigma_x\rangle_{\beta_c}\langle \sigma_y\sigma_{\mathcal{R}_n(y)}\rangle_{\beta_c}
    \\
    &\lesssim 
    n^{2-d}\sum_{x\in \Lambda_{2n}}\langle \sigma_0\sigma_x\rangle_{\beta_c}\lesssim n^{4-d}.
      \end{aligned}
 \end{align}
If $\lfloor 3n/4\rfloor <x_1\leq n$, using the spectral representation of these models and the Messager--Miracle-Sol\'e inequalities \cite{MessagerMiracleSoleInequalityIsing}, we obtain the following gradient estimate
\begin{equation}\label{eq: gradient estimate}
    \langle \sigma_{u'}\sigma_x\rangle_{\beta_c}-\langle \sigma_{u'}\sigma_{\mathcal{R}_n(x)}\rangle_{\beta_c}
    \leq 
    C_1\frac{|x-\mathcal{R}_n(x)|}{n}\langle \sigma_0\sigma_{\lfloor n/4\rfloor \mathbf{e}_1}\rangle_{\beta_c}.
\end{equation}
Using \eqref{eq: gradient estimate} and \eqref{eq:IRB}, we get
\begin{multline}\label{eq:proof of prop 3.6}
    \sum_{\substack{x,y\in \Lambda_n\\ x_1> \lfloor 3n/4\rfloor \\y\sim x}}\Big(\langle \sigma_{u'}\sigma_x\rangle_{\beta_c}-\langle \sigma_{u'}\sigma_{\mathcal{R}_n(x)}\rangle_{\beta_c}\Big)\langle \sigma_y\sigma_{\mathcal{R}_n(y)}\rangle_{\beta_c}
    \\\lesssim n^{1-d}\sum_{\substack{y\in \Lambda_n\\ \:y_1\geq \lfloor 3n/4\rfloor }}|y-\mathcal R_n(y)|\langle \sigma_y\sigma_{\mathcal R_n(y)}\rangle_{\beta_c}
    \lesssim \sum_{1\leq k\leq n}k\langle \sigma_0\sigma_{k\mathbf{e_1}}\rangle_{\beta_c}\lesssim (\log n)^{\mathds{1}_{d=4}}.
    \end{multline}
Plugging \eqref{eq:proof of prop 3.3} and \eqref{eq:proof of prop 3.6} into \eqref{eq:proof of prop 2} concludes the proof.
		\end{proof}
\begin{Rem} One can check that \eqref{eq:proof of prop 3} does not hold if $0$ is replaced by $u$ that is ``too close'' to $\mathbb H_n$. This is the main reason why we switched $\beta$ to $\beta'$ only for edges within $\Lambda_{n/2}$ in \eqref{eq: def J}.
\end{Rem}

We are now in a position to prove Proposition \ref{prop:properties of the random set}. Recall the definition of $J$ from \eqref{eq: def J}.
\begin{proof}[Proof of Proposition~\textup{\ref{prop:properties of the random set}}] Let $d\geq 4$, $(\beta,h)\in \mathcal Q$ and $\beta'\in [0,\beta)$.

We begin with the proof of \eqref{eq:prop 1}. By \eqref{eq:bound of proba 0 notin S_n}, one has
\begin{equation}\label{eq:proof of prop 1}
	\mathbf P_{\beta_c}^\emptyset [0\in \mathcal S_n]\geq 1-2d\langle \sigma_0\sigma_{2n\mathbf{e}_1}\rangle_{\beta_c}\geq 1-\frac{C_1}{1\vee (2n)^{d-2}},
\end{equation} 
where $C_1=C_1(d)>0$ and where we used \eqref{eq:IRB} in the second inequality. The result follows by choosing $n$ large enough.

Next, observe that \eqref{eq:prop 3} is exactly Corollary \ref{coro:bound translate random s_n}.
	
	Finally, we prove \eqref{eq:prop 2}. Let $\varepsilon>0$ to be chosen small enough below. First, 
	\begin{equation}\label{eq:proof of prop 4}
		\mathbf E_{\beta_c}^\emptyset\Big[\sum_{\substack{x\in \mathcal S_n\cap \Lambda_{\varepsilon n}\\y\notin \mathcal S_n, \: y\sim x}}\langle \sigma_0\sigma_x\rangle_{\mathcal S_n,J(0)}J_{xy}(0)\Big]\leq  \beta\sum_{\substack{x\in \Lambda_{\varepsilon n}\\y\sim x}}\langle \sigma_0\sigma_x\rangle_{\beta_c}\langle \sigma_y\sigma_{\mathcal R_n(y)}\rangle_{\beta_c}\lesssim (\varepsilon n)^2 n^{2-d}\leq \varepsilon^2 ,
	\end{equation}
		where we used \eqref{eq:switching where we move 0} and Griffiths' inequality in the first inequality, and \eqref{eq:IRB} in the second one. Fix $\varepsilon\in (0,1/2)$ small enough so that the right-hand side above is at most $\tfrac{1}{16}$. Now, recall the definition of $L(\beta')$ from \eqref{eq:def L(beta)}, and fix a set $S\subset \Lambda_{L(\beta')}$ such that $\varphi_{\beta'}(S)\leq \frac{1}{4}$. We will prove that there exists $c>0$ such that, for $n$ large enough, 
	\begin{equation}\label{eq:proof of prop 5}
		\mathbf E_{\beta_c}^\emptyset\Big[\sum_{\substack{x\in \mathcal S_n\setminus\Lambda_{\varepsilon n}\\y\notin \mathcal S_n, \: y\sim x}}\langle \sigma_0\sigma_x\rangle_{\mathcal S_n,J(0)}J_{xy}(0)\Big]\lesssim \exp\left(-c\frac{\varepsilon n}{L(\beta')}\right)(\log n)^{\mathds{1}_{d=4}}.
	\end{equation}
		The result is clear if $L(\beta')\geq \tfrac{\varepsilon n}{10}$ (by Corollary \ref{coro:bound translate random s_n}). We therefore assume that $L(\beta')\leq \tfrac{\varepsilon n}{10}$. To prove \eqref{eq:proof of prop 5}, we rely on the Simon--Lieb inequality stated in Lemma \ref{lem:SL}. Iterating $k$ times Lemma \ref{lem:SL} with $\Lambda=\mathcal S_n$ and translates of $\mathcal S_n\cap S$, where $k:=\lfloor\tfrac{\varepsilon n}{2(L(\beta')+1)} \rfloor$ gives
	\begin{multline}
		\sum_{\substack{x\in \mathcal S_n\setminus\Lambda_{\varepsilon n}\\y\notin \mathcal S_n, \: y\sim x}}\langle \sigma_0\sigma_x\rangle_{\mathcal S_n,J(0)}J_{xy}(0)\leq \sum_{\substack{u_1\in S\\v_1\notin S, \: v_1\sim u_1}}\langle \sigma_0\sigma_{u_1}\rangle_{S,\beta'} \beta' \\\cdots \sum_{\substack{u_k\in S+v_{k-1}\\v_k\notin S+v_{k-1},\: v_k\sim u_k}}\langle \sigma_{v_{k-1}}\sigma_{u_k}\rangle_{S+v_{k-1},\beta'}\beta' \sum_{\substack{x\in \mathcal S_n\\y\notin \mathcal S_n, \: y\sim x}}\langle \sigma_{v_k}\sigma_x\rangle_{\mathcal S_n,\beta}J_{xy}(0),
		\end{multline}
	where we used that $\Lambda_{k(L(\beta')+1)}\subset \Lambda_{\varepsilon n}$, Griffiths' inequality to replace $\mathcal S_n\cap (S+v_i)$ by $S+v_i$ $(0\leq i \leq k-1)$ above, and the fact that $J(0)$ is equal to $\beta'$ on the box $\Lambda_{n/2}$ which contains $\Lambda_{k(L(\beta')+1)}$. Averaging $\mathcal S_n$ with respect to $\mathbf E^\emptyset_{\beta_c}$ and using Corollary \ref{coro:bound translate random s_n} (observe that $v_k\in \Lambda_{n/2}$) yields
	\begin{equation}\label{eq:proof of prop 6}
		\mathbf E_{\beta_c}^{\emptyset}\Big[\sum_{\substack{x\in \mathcal S_n\setminus\Lambda_{\varepsilon n}\\y\notin \mathcal S_n, \: y\sim x}}\langle \sigma_0\sigma_x\rangle_{\mathcal S_n,J(0)}J_{xy}(0)\Big]\lesssim 4^{-k} (\log n)^{\mathds{1}_{d=4}}.
	\end{equation}
	Combining \eqref{eq:proof of prop 4} (recall how we chose $\varepsilon$) and \eqref{eq:proof of prop 6} gives the existence of $c_1=c_1(\varepsilon,d)>0$ and $C_1=C_1(d)>0$ such that,
	\begin{equation}
		\mathbf E_{\beta_c}^\emptyset[\varphi_{J(0)}(\mathcal S_n)]\leq \frac{1}{16}+C_1\exp\left(-c_1\frac{n}{L(\beta')}\right)\cdot (\log n)^{\mathds{1}_{d=4}}.
	\end{equation}
	The desired result then follows from Theorem \ref{thm:sharplength}.
	\end{proof}

\section{Proof of the lower bounds on $\langle \sigma_0;\sigma_x\rangle_{\beta,h}$}\label{sec:lower bounds}
In this section, we prove the lower bounds in Theorems~\ref{thm:main} 
and~\ref{thm:main bis}. The argument is based on differentiation in $h$: 
we compare $\langle\sigma_0;\sigma_x\rangle_{\beta,h}$ to 
$\langle\sigma_0\sigma_x\rangle_{\beta}$ by controlling the derivative of 
the truncated two-point function with respect to $h$ via 
Proposition \ref{prop: sum of K}, and by lower 
bounding the two-point function at $h=0$ via Theorems \ref{thm:bounds 2pt d>4} and \ref{thm:bounds 2pt d=4}. A key element of the proof is the following bound, which can be extracted\footnote{Observe that the $\beta$ multiplying $h$ in the $\tanh$ term in \cite{AizenmanFernandezCriticalBehaviorMagnetization1986} has disappeared here due to our different parametrisation of the Hamiltonian.} from the proof of \cite[Theorem~5.7]{AizenmanFernandezCriticalBehaviorMagnetization1986} (see display after (5.36) there):
\begin{multline}\label{eq: lb on derivative in h of 2pt}
	0\geq \frac{\partial}{\partial h}\langle \sigma_0;\sigma_x\rangle_{\beta,h}\geq -2\sum_{y,z\in \mathbb Z^d}\Big([\tanh(h)+2d\beta M(\beta,h)]K_{\beta,h}(0,z)K_{\beta,h}(x,z)K_{\beta,h}(y,z)\\+\beta M(\beta,h)\sum_{z'\sim z}K_{\beta,h}(0,z')K_{\beta,h}(x,z')K_{\beta,h}(y,z)\Big).
\end{multline}

\begin{proof}[Proof of the lower bound in Theorem~\textup{\ref{thm:main}}] Let $d>4$. Let $(\beta,h)\in \mathcal Q^*$. Combining Theorem \ref{thm: magnetisation} and Proposition \ref{prop: sum of K} gives, for every $s\in (0,h]$,
\begin{multline}\label{eq:pf lb 1}
	\sum_{y,z\in \mathbb Z^d}[\tanh(s)+2d\beta M(\beta,s)]K_{\beta,s}(0,z)K_{\beta,s}(x,z)K_{\beta,s}(y,z)\\\lesssim s^{1/3}\cdot s^{-1}s^{1/3} \cdot \sum_{z\in \mathbb Z^d}\langle \sigma_0\sigma_z\rangle_{\beta}\langle \sigma_z\sigma_x\rangle_{\beta}.
\end{multline}
Using Griffiths' inequality and \eqref{eq:IRB} yields
\begin{equation}\label{eq:pf lb 2}
	\sum_{z\in \mathbb Z^d}\langle \sigma_0\sigma_z\rangle_{\beta}\langle \sigma_z\sigma_x\rangle_{\beta}\lesssim \sum_{z\in \mathbb Z^d}\frac{1}{1\vee |z|^{d-2}}\frac{1}{1\vee |x-z|^{d-2}}\lesssim \frac{1}{1\vee |x|^{d-4}}.
\end{equation}
Plugging \eqref{eq:pf lb 2} into \eqref{eq:pf lb 1} gives
\begin{equation}\label{eq:pf lb 3}
	\textup{LHS of \eqref{eq:pf lb 1}}\lesssim \frac{s^{-1/3}}{1\vee |x|^{d-4}}. 
\end{equation}
A similar computation gives, for every $s\in(0,h]$,
\begin{equation}\label{eq:pf lb 4}
	\beta M(\beta,s)\sum_{y,z\in \mathbb Z^d}\sum_{z'\sim z}K_{\beta,s}(0,z')K_{\beta,s}(x,z')K_{\beta,s}(y,z)\lesssim \frac{s^{-1/3}}{1\vee |x|^{d-4}}.
\end{equation}
Using the fundamental theorem of calculus, \eqref{eq: lb on derivative in h of 2pt}, and \eqref{eq:pf lb 3}--\eqref{eq:pf lb 4} gives $C_1=C_1(d)>0$ such that, for every $x\in \mathbb Z^d$,
\begin{equation}\label{eq:pf lb 5}
	\langle \sigma_0;\sigma_x\rangle_{\beta,h}=\langle \sigma_0\sigma_x\rangle_\beta+\int_0^h \frac{\partial}{\partial s}\langle \sigma_0;\sigma_x\rangle_{\beta,s}\mathrm{d}s\geq \langle \sigma_0\sigma_x\rangle_{\beta}-\frac{C_1 h^{2/3}}{1\vee |x|^{d-4}}.
\end{equation}
By Theorem \ref{thm:bounds 2pt d>4}, there exists $c_1>0$ such that, for $|x|\leq c_1\ell(\beta,0)$, $\langle \sigma_0\sigma_x\rangle_{\beta}\geq c_1(1\vee |x|)^{2-d}$. Now choose $c_2= c_1\wedge\sqrt{c_1/(2C_1)}$. If $1\leq |x|\leq c_2\ell(\beta,h)\leq c_2h^{-1/3}$, then
\begin{equation}
	\frac{C_1 h^{2/3}}{1\vee |x|^{d-4}}\leq \frac{(c_1/2)}{1\vee |x|^{d-2}}.
\end{equation}
As a result, if $|x|\leq c_1\ell(\beta,0)\wedge c_2\ell(\beta,h)=c_2\ell(\beta,h)$,
\begin{equation}
	\langle \sigma_0;\sigma_x\rangle_{\beta,h}\geq \frac{1}{2}\langle \sigma_0\sigma_x\rangle_{\beta}\gtrsim \frac{1}{1\vee |x|^{d-2}},
\end{equation}
which concludes the proof.
\end{proof}

A similar argument can be used when $d=4$, but more care is needed as the upper bound in \eqref{eq:pf lb 3} is now infinite. To solve this issue, we rely on the following result.
\begin{Lem}\label{lem:bubble of K} Let $d=4$. There exist $C,\delta>0$ such that, if $(\beta,h)\in \mathcal Q_\delta$,
\begin{equation}
	\sum_{z\in \mathbb Z^4} K_{\beta,h}(0,z)^2\leq C\log \ell(\beta,h).
\end{equation}
\end{Lem} 
\begin{proof} Let $\varepsilon=1$. Let $C_1,\delta>0$ be given by Theorem \ref{thm:finding the good random S d=4} and $c_2,C_2>0$ be given Proposition \ref{prop:sufficient condition to exp decay random}. By the same results, if $(\beta,h)\in \mathcal Q_\delta$, for every $x\in \mathbb Z^4$,
\begin{equation}
	K_{\beta,h}(0,x)\leq \frac{C_2}{1\vee |x|^2}\exp\left(-(c_2/C_1)\frac{|x|}{\ell(\beta,h)(\log \ell(\beta,h))^{3}}\right).
\end{equation}
With this bound, we find
\begin{equation}
	\sum_{z\in \mathbb Z^4} K_{\beta,h}(0,z)^2\lesssim \sum_{k\geq 1}\frac{1}{k}\exp\left(-(c_2/C_1)\frac{k}{\ell(\beta,h)(\log \ell(\beta,h))^{3}}\right)\lesssim \log \ell(\beta,h).
\end{equation}
This concludes the proof.
\end{proof}

\begin{proof}[Proof of the lower bound in Theorem~\textup{\ref{thm:main bis}}] We reproduce the argument used in dimensions $d>4$. By Theorem \ref{thm: magnetisation} and Proposition \ref{prop: sum of K}, for every $s\in (0,h]$,
\begin{multline}\label{eq:pf lb 1 d=4}
	\sum_{y,z\in \mathbb Z^4}[\tanh(s)+2d\beta M(\beta,s)]K_{\beta,s}(0,z)K_{\beta,s}(x,z)K_{\beta,s}(y,z)\\\lesssim s^{1/3}|\log s|\cdot s^{-1} s^{1/3}|\log s| \cdot \sum_{z\in \mathbb Z^4}K_{\beta,s}(0,z)K_{\beta,s}(x,z).
\end{multline}
Now, by the Cauchy--Schwarz inequality and Lemma \ref{lem:bubble of K},
\begin{equation}\label{eq:pf lb 2 d=4}
	\sum_{z\in \mathbb Z^4}K_{\beta,s}(0,z)K_{\beta,s}(x,z)\leq \sum_{z\in \mathbb Z^4} K_{\beta,s}(0,z)^2\lesssim \log \ell(\beta,s)\lesssim |\log s|.
\end{equation}
Combining \eqref{eq:pf lb 2 d=4} and \eqref{eq:pf lb 1 d=4} gives
\begin{equation}\label{eq:pf lb 3 d=4}
	\textup{LHS of \eqref{eq:pf lb 1 d=4}}\lesssim s^{-1/3}|\log s|^3.
\end{equation}
A similar computation gives, for every $s\in(0,h]$,
\begin{equation}\label{eq:pf lb 4 d=4}
	\beta M(\beta,s)\sum_{y,z\in \mathbb Z^4}\sum_{z'\sim z}K_{\beta,s}(0,z')K_{\beta,s}(x,z')K_{\beta,s}(y,z)\lesssim s^{-1/3}|\log s|^3.
\end{equation}
Using the fundamental theorem of calculus, \eqref{eq: lb on derivative in h of 2pt}, and \eqref{eq:pf lb 3 d=4}--\eqref{eq:pf lb 4 d=4} gives $C_1>0$ such that, for every $x\in \mathbb Z^4$,
\begin{equation}\label{eq:pf lb 5 d=4}
	\langle \sigma_0;\sigma_x\rangle_{\beta,h}=\langle \sigma_0\sigma_x\rangle_\beta+\int_0^h \frac{\partial}{\partial s}\langle \sigma_0;\sigma_x\rangle_{\beta,s}\mathrm{d}s\geq \langle \sigma_0\sigma_x\rangle_{\beta}-C_1 h^{2/3}|\log h|^3.
\end{equation}
By Theorem \ref{thm:bounds 2pt d=4}, there exists $c_1>0$ such that, for $|x|\leq c_1\ell(\beta,0)$, 
\begin{equation}
\langle \sigma_0\sigma_x\rangle_{\beta}\geq \frac{c_1}{1\vee |x|^2\log |x|}.
\end{equation}
Moreover, if $|x|^2\log |x|\leq (h^{2/3}|\log h|^3)^{-1}$, or alternatively $|x|\leq c_2 h^{-1/3}|\log h|^{-2}$ for $c_2>0$ small enough,
\begin{equation}
	C_1 h^{2/3}|\log h|^3\leq \frac{(c_1/2)}{1\vee |x|^2\log |x|}.
\end{equation}
As a result, if $|x|\leq c_1\ell(\beta,0)\wedge c_2 h^{-1/3}|\log h|^{-2}$,
\begin{equation}
	\langle \sigma_0;\sigma_x\rangle_{\beta,h}\geq \frac{1}{2}\langle \sigma_0\sigma_x\rangle_{\beta}\gtrsim \frac{1}{1\vee |x|^2\log |x|}.
\end{equation}
Since $\ell(\beta,h)(\log \ell(\beta,h))^{-2}\lesssim c_1\ell(\beta,0)\wedge c_2 h^{-1/3}|\log h|^{-2}$, this concludes the proof.
\end{proof}

\bibliographystyle{plain}
\bibliography{biblio_sorted.bib}
\end{document}